\documentclass[USenglish]{article}
\usepackage[utf8]{inputenc}
\usepackage[T1]{fontenc}
\usepackage{amsmath, amssymb, amsthm, amscd, amsfonts, graphicx, fancyhdr, color, mathrsfs,tikz-cd,tikz,tikz, comment, verbatim, fancyvrb}
\usetikzlibrary{positioning,automata,shapes,decorations}
\usepackage{latexsym}
\usepackage{url, multicol, hyperref}
\usepackage{subfig}
\usepackage{multirow}
\usepackage{geometry}                
\newtheorem{theorem}{Theorem}[section]
\newtheorem{conjecture}[theorem]{Conjecture}
\newtheorem{lemma}[theorem]{Lemma}
\newtheorem{proposition}[theorem]{Proposition}
\newtheorem{corollary}[theorem]{Corollary}

\theoremstyle{definition}

\theoremstyle{definition}

\theoremstyle{definition}
\newtheorem{remark}[theorem]{Remark}
\theoremstyle{definition}
\newtheorem{example}[theorem]{Example}
\theoremstyle{definition}
\newtheorem{definition}[theorem]{Definition}
\theoremstyle{definition}

\theoremstyle{definition}

\theoremstyle{definition}
\newtheorem{question}[theorem]{Question}

\newtheorem*{thm1}{Theorem \ref{thm:saddles}}
\newtheorem*{thm2}{Theorem \ref{thm:abcd}}
\newtheorem*{thm3}{Theorem \ref{thm:double}}

\newcommand{\fv}[4]{\begin{bmatrix} #1  \\ #2 \\ #3 \\ #4 \end{bmatrix}} 
\newcommand{\sfv}[4]{\left[\begin{smallmatrix} #1 \\ #2 \\ #3 \\ #4 \end{smallmatrix}\right]} 
\newcommand{\bv}[2]{\begin{bmatrix} #1  \\ #2 \end{bmatrix}} 
\newcommand{\sv}[2]{\left[\begin{smallmatrix} #1  \\ #2 \end{smallmatrix}\right]} 
\newcommand{\bhv}[2]{\begin{bmatrix} #1  & #2 \end{bmatrix}} 
\newcommand{\shv}[2]{\left[\begin{smallmatrix} #1  & #2 \end{smallmatrix}\right]} 
\newcommand{\bm}[4]{\begin{bmatrix} #1 & #2 \\ #3 & #4 \end{bmatrix}} 
\newcommand{\sm}[4]{\left[\begin{smallmatrix} #1 & #2 \\ #3 & #4 \end{smallmatrix}\right]} 
\newcommand{\ar}[2]{{#1}^{#2}}

\newcommand{\vv}{\mathbf{v}}

\newcommand{\R}{\mathbf{R}}
\newcommand{\Q}{\mathbf{Q}}
\newcommand{\N}{\mathbf{N}}
\newcommand{\Z}{\mathbf{Z}}
\newcommand{\D}{\mathfrak{D}}
\newcommand{\bs}{\mathbf{s}}
\newcommand{\qrf}{\mathbf{Q}(\sqrt 5)}
\newcommand{\fti}{{\bf(FILL THIS IN)}}
\newcommand{\h}{\mathcal{H}}

\newcommand{\1}{\mathbf{1}}
\newcommand{\2}{\mathbf{2}}
\newcommand{\3}{\mathbf{3}}
\newcommand{\4}{\mathbf{4}}
\newcommand{\5}{\mathbf{5}}

\newcommand{\etalchar}[1]{$^{#1}$}

\renewcommand\hat\widehat

\renewcommand\floatpagefraction{.9}
\renewcommand\topfraction{.9}
\renewcommand\bottomfraction{.9}
\renewcommand\textfraction{.1}
\title{Periodic paths on the pentagon, \\ double pentagon and golden L}
\author{Diana Davis and Samuel Lelièvre}

\begin{document}
\maketitle


\pagestyle{myheadings}


\begin{abstract}

  We give a tree structure on the set of all periodic directions on the golden L, which gives an associated tree structure on the set of periodic directions for the pentagon billiard table and double pentagon surface. We use this to give the periods of periodic directions on the pentagon and double pentagon. We also show examples of many periodic billiard trajectories on the pentagon, which are strikingly beautiful, and we describe some of their properties. Finally, we give conjectures and future directions based on experimental computer evidence.
\end{abstract}

\section{Main results, introduction and background} \label{sec:intro}

\subsection{Periodic trajectories in polygons: existence, abundance, enumeration}

One problem in billiards is whether every polygonal billiard table has a periodic path. This question is open and of great interest even for triangles, where Richard Schwartz showed that every triangle whose largest angle is less than $100^\circ$ has a periodic path, and this result was recently extended by four authors to $112.3^\circ$ \cite{100, twelve}. From Howard Masur's result we know that every rational polygon has a periodic path, and in fact countably many \cite{masur}. That every \emph{regular} polygon has a periodic path is clear; a path connecting midpoints of any two edges extends to a periodic path.

One specific family of polygons is those that tile the plane by
reflections: the square, rectangle, equilateral triangle, hexagon,
half-equilateral triangle, and isosceles right triangle. For those,
because of the possibility to develop to the plane, it is easy to
describe the countable set of directions in which there
are periodic paths, and to enumerate those periodic paths, as we
recall in \S\ref{subsec:square} for the case of the square.

\begin{figure}[!ht]
  \centering
  \includegraphics[width=0.24\textwidth]{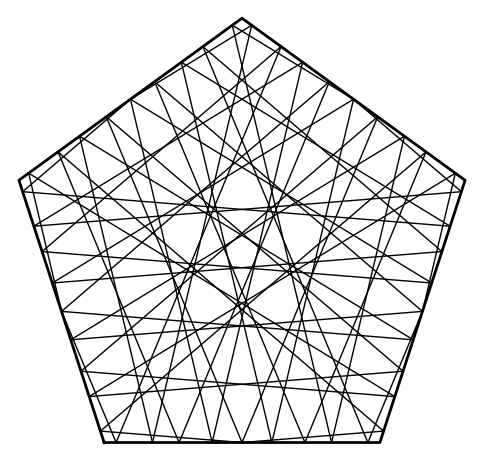} \
  \includegraphics[width=0.24\textwidth]{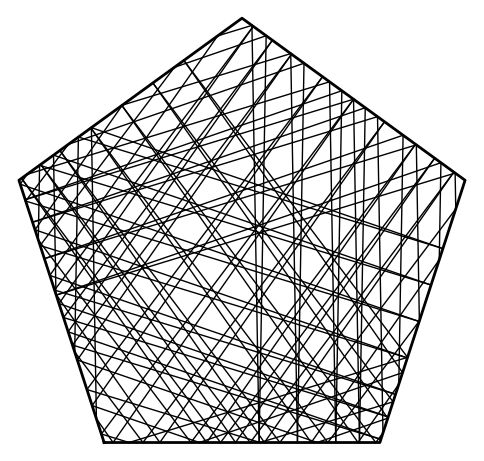}\
  \includegraphics[width=0.24\textwidth]{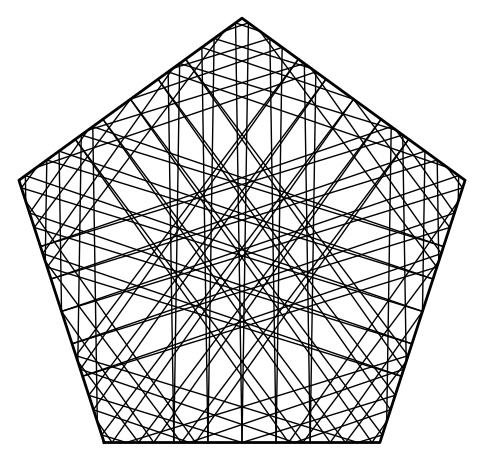}\
  \includegraphics[width=0.24\textwidth]{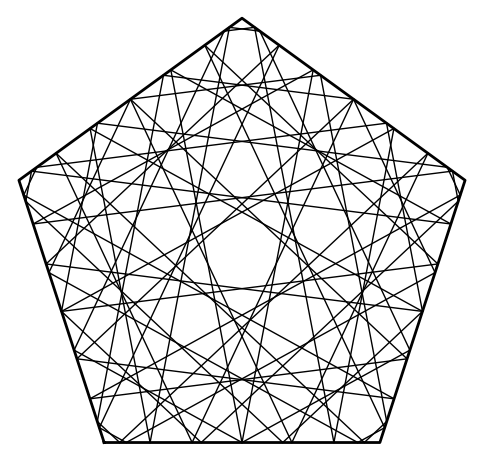}
  \caption{Some periodic paths on the regular pentagon. These are the trajectories 1000-short, 1231-long, 102-short, and 133-short, respectively (see \S\ref{sec:tree}). \label{fig:nice_examples}}
\end{figure}

In a sense, the next interesting polygons to look at are non-tiling regular polygons and non-tiling rational triangles, where existence
is known but enumeration is new. Here we enumerate periodic paths on
the regular pentagon, which allows us to produce pictures of striking
beauty. Several examples of such paths appear above (Figure \ref{fig:nice_examples}).

\subsection{Main results}

To study billiard paths on polygons, an essential tool is \emph{unfolding} the billiard table into a surface. For the regular pentagon billiard table, the associated surface is a \emph{necklace} made of 10 regular pentagons, and the related \emph{double pentagon} surface made of two regular pentagons, each with oppositely-oriented parallel edges identified (Figure \ref{fig:necklace}). The double pentagon is related to a third translation surface, the \emph{golden L}, via a shear and a cut and paste (Figure \ref{fig:cutandpaste}). Because the golden L is made of rectangles and is thus easier to use than a regular pentagon, we use the golden L surface in this paper for all of our theory and computations, and then we translate them into our desired results for the regular pentagon surface and billiard table.

The golden L is a right-angled, L-shaped translation surface obtained by gluing together
a $1\times 1$ square and two rectangles, each
with a long side of length 1 and a short side $\phi$ times shorter,
where $\phi$ is the golden ratio, that is, the positive root of
the polynomial $x^2 - x - 1$. (Figure \ref{fig:golden-L-color} and Definition \ref{def:goldenL}). We give a description of all saddle connection vectors on the golden L in a quaternary \emph{tree} analogous to the Farey tree; this tree is our main tool.

\begin{figure}[!ht]
  \begin{center}
    \includegraphics[width=0.3\textwidth]{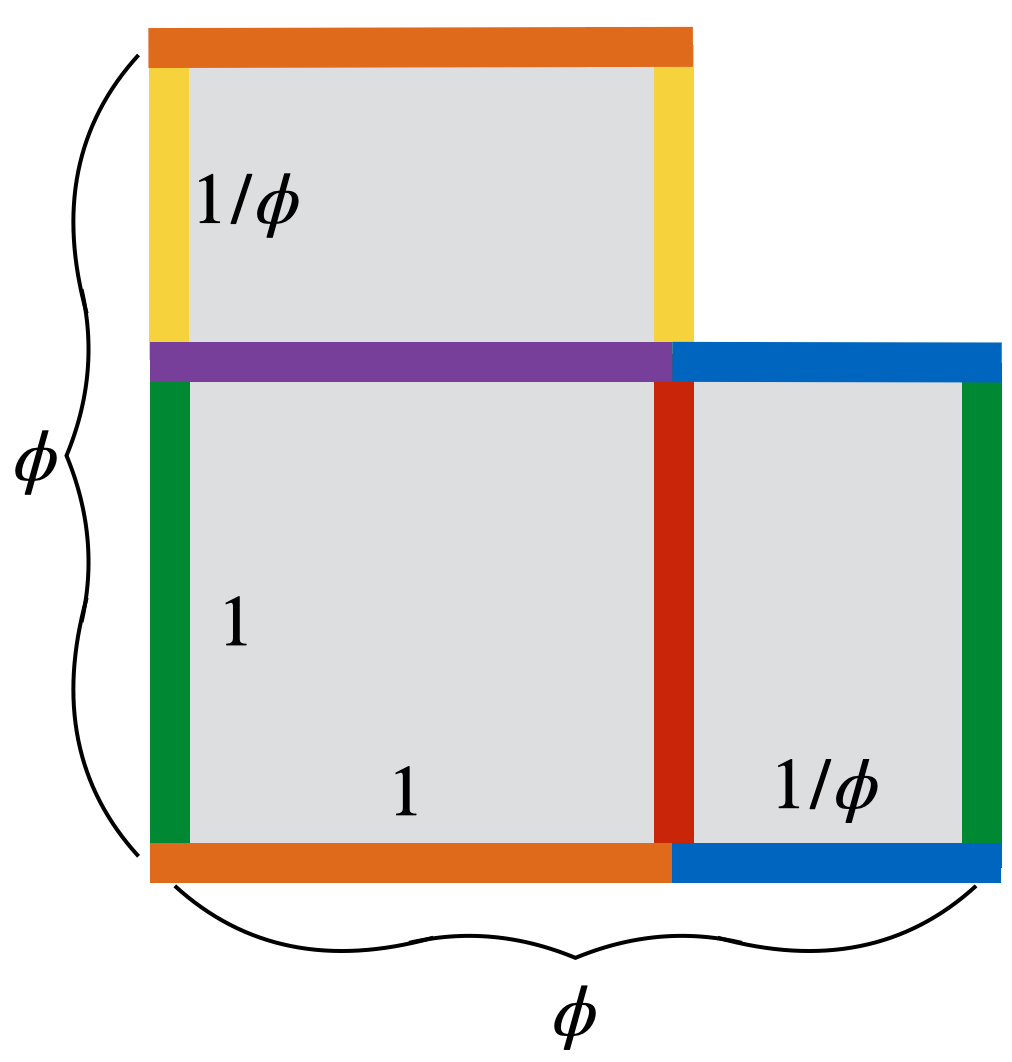}
    \caption{The golden L translation surface, with edge identifications indicated by color, and edge lengths as indicated in the figure \label{fig:golden-L-color}}
  \end{center}
\end{figure}

We define four matrices
$\sigma_0 = \begin{bmatrix}1 & \phi \\ 0 & 1 \\ \end{bmatrix}$,
$\sigma_1 = \begin{bmatrix}\phi & \phi \\ 1 & \phi \\ \end{bmatrix}$,
$\sigma_2 = \begin{bmatrix}\phi & 1 \\ \phi & \phi \\ \end{bmatrix}$,
$\sigma_3 = \begin{bmatrix}1 & 0 \\ \phi & 1 \\ \end{bmatrix}$.

We denote by $\Sigma$ the set of nonzero vectors
$\begin{bmatrix}x \\ y \\ \end{bmatrix}\in \R^2$
with $x > 0$ and $y \ge 0$, which we call the \emph{positive cone}.
The four sectors $\Sigma_i = \sigma_i \Sigma$ for $i \in \{0, 1, 2, 3\}$
form a partition of $\Sigma$.

\begin{thm1}
  Let $A^*$ be the set of all finite words on the alphabet $A=\{\sigma_0, \sigma_1, \sigma_2, \sigma_3\}$ of linear maps taking the positive cone to its four partitioning sectors (Definition \ref{def:sigmas}). Then the set of all vectors $\sigma\sv 10$, for $\sigma\in A^*$, gives the set of short cylinder vectors in the positive cone.
\end{thm1}

A periodic direction on the golden L is a direction with slope in $\mathbf{Q}(\sqrt 5)$ \cite{arnoux}. There are, of course, many ways to express such a slope as a vector. For the canonical vectors, we use \emph{saddle connections}: line segments whose endpoints are both at a singularity (``corner'') of the surface, where the cone angle is $6\pi$. For example, in Figure \ref{fig:golden-L-color}, we call the horizontal blue saddle connection the ``short,'' and the orange and purple horizontal the ``long'' saddle connections (also see Figure \ref{fig:saddle-connections}).  We use the long saddle connection vector in each periodic direction; this is the vector obtained from our algorithm to generate the tree.

\begin{thm2}
  Corresponding to any periodic direction vector $\vv$ in the first quadrant on the golden L is a unique sequence \mbox{$a_\vv = (k_0, k_1, ..., k_n)$} of sectors as in Definition~\ref{def:algorithm}, such that $\vv= \ell_\vv \sigma_{k_0} \ldots \sigma_{k_n} \sv{1}{0}$ for some length $\ell_\vv$. Then the long saddle connection vector in the direction $\mathbf{v}$ has $\ell_\vv=1$ and is thus $\sigma_{k_0}\ldots\sigma_{k_n}\sv 10$. Furthermore, the long saddle connection vector is of the form $\sv{a+b\phi}{c+d\phi}$ for $a,b,c,d\in\N$.
\end{thm2}

We use the saddle connection vector in this canonical representation to determine the combinatorial period (number of edge crossings before repeating) of the associated double pentagon trajectory:

\begin{thm3}In each periodic direction, there are two periodic trajectories, whose corresponding direction in the golden L has long saddle connection vector $\sv{a+b\phi}{c+d\phi}$ as above. The short trajectory in this direction has holonomy vector $\vv = \sv{a+b\phi}{c+d\phi}$ in the golden L and combinatorial period $2(a+b+c+d)$ in the double pentagon, and the long trajectory in this direction has holonomy vector $\phi\vv$ in the golden L and combinatorial period $2(a+2b+c+2d)$ in the double pentagon.
\end{thm3}

This result for the double pentagon surface is directly analogous to the corresponding result on the square billiard table, where a trajectory with slope $p/q$ has associated holonomy vector $[q,p]$ and period $2(p+q)$; see \S \ref{sec:periods}.

\subsection{Analogous known results for the square}\label{subsec:square}
For a rational billiard table, the set of periodic directions is countable and dense \cite{masur}. In this paper, we describe the structure of this set for the pentagon and the golden L.
The structure of the set of periodic directions is well understood for the square torus and square billiard table. We summarize those results here, as we will extend them to the golden L surface, the double pentagon surface and the pentagon billiard table.

\begin{theorem} \label{thm:square}
  For the square torus and square billiard table:
  \begin{enumerate}
    \item The periodic directions on the square torus and  billiard table are those with rational slope.
    \item A trajectory with slope $a/b$ (in lowest terms) on the square torus has combinatorial period $a+b$, and on the square billiard table has combinatorial period $2(a+b)$.
  \end{enumerate}
\end{theorem}

\begin{proof}
  We assume $a,b\in\mathbf{N}$ and that $a/b$ is in lowest terms, and we assume as usual that the trajectory does not hit a vertex.
  \begin{enumerate}
    \item  A trajectory with rational slope meets only a finite number of points on the edges of the table, so it must be periodic; unfolding a periodic path yields a trajectory with rational slope.
    \item Developing the square torus to a grid of squares, we can see that a line of slope $a/b$ on the square grid repeats itself after going up $a$ squares (down if $a<0$) and to the right $b$ squares (left if $b<0$). The number of edges crossed is thus $a+b$.
          Unfolding the square billiard table to the square torus requires $4$ copies of the table inside each torus, which is equivalent to adding grid lines at half integers, so the path repeats after $2(a+b)$ edge crossings of these lines.
  \end{enumerate}
\end{proof}
For more discussion and for illustrations, see \cite{flatsurfaces}, \S 4.

\subsection{Previous work on the pentagon, double pentagon and golden L}\label{sec:previous}

William Veech studied billiards in regular polygons, including the regular pentagon \cite{veech}. Curt McMullen characterized the L-shaped polygons that are lattice (Veech) surfaces; one such surface is the golden L, an example given in that paper \cite[Theorem 9.2 and Figure 4]{curt}.

In \cite{DFT}, the first author, Dmitry Fuchs and Sergei Tabachnikov explored periodic directions on the double pentagon surface and pentagon billiard table, as we do here. There is some overlap between that paper and the current paper, but we use different methods and have a different emphasis. In particular, in that paper the hyperbolic plane in the Poincar\'e model is used as the basis for all of the direction computations on the pentagon, and the results on periodic directions are given in these terms. By contrast, in the current paper we use translation surfaces, in particular the golden L, as the basis for our direction computations. We obtain several of the same results, using new translation surface methods $-$ our Theorem \ref{thm:saddles} characterizes periodic directions using a Farey-like tree based on the golden L, as does their Theorem 1 using the hyperbolic plane; their matrices $X_i$ in \S 2.2 are the same as our matrices $\sigma_i$ in Definition \ref{def:sigmas}; their generating transition diagrams in equation (1) and ``enhanced graphs'' in Figure 14 are the same as our transition diagrams (under a permutation on edge labels) in Lemma \ref{lem:dptd}; their Theorems 7 and 10 combine to give a result analogous to our Combinatorial Period Theorem \ref{thm:double}, but our methods for that Theorem are more direct.

The following results are new:  Theorems~\ref{thm:saddles} and \ref{thm:abcd} characterizing saddle connections on the golden L, Theorem~\ref{thm:tree_symm} about the symmetry of the tree of directions, Corollary~\ref{cor:palindrome} about the structure of cutting sequences on the double pentagon, \S\ref{sec:veech_group}-\ref{sec:symm_group} and Theorem~\ref{thm:symm} about the Veech group of the necklace surface and the applications to the symmetries of trajectories on the pentagon billiard, Corollary~\ref{cor:sixths} giving the proportion of paths with and without rotational symmetry, Theorem~\ref{thm:geom-symm} and Proposition~\ref{prop:corecurve} giving details about how different types of paths arise, Propositions~\ref{prop:evens} and \ref{prop:10} about the periods that arise in the double pentagon and the pentagon billiard, and \S\ref{sec:non-equi} about non-equidistribution of certain periodic billiard paths. Furthermore, the pictures from our \verb|Sage| program allow us to look at periodic trajectories on the pentagon for the first time, which are strikingly beautiful and afford us additional insights about the system.

The second author, Jayadev Athreya and Jon Chaika studied slope gaps on the golden L, which is the same surface that we use here, but with an entirely different focus \cite{acl}.

\subsection{Related work}\label{sec:related}

Much is known about geodesics on regular polygons and their associated combinatorics. John Smillie and Corinna Ulcigrai \cite{SU,SU2} characterize all cutting sequences corresponding to geodesics on the regular octagon surface and extend their results to all regular polygons with an even number of edges. The first author \cite{davis13} did similarly for the double regular pentagon and all surfaces made from two regular polygons with an odd number of sides. Subsequently Davis and Ulcigrai, along with Irene Pasquinelli, extended this work to the Bouw-Möller surfaces, a large family of Veech surfaces made from $m\geq 2$ semi-regular $2m$-gons \cite{DPU}. The present paper characterizes all periodic directions on the double regular pentagon, and gives substitution rules (Definition \ref{def:decorated} and \S~3) for augmenting a cutting sequence.

After reading this paper, Diaaeldin Taha wrote a paper \cite{T18} in which he defines a tree structure on the orbit of the vector $[1,0]$ under the linear action of each Hecke group $G_q$, for $q\geq 3$ (\S~2.1-2.2). That generalizes the tree structure we define here, which corresponds to the case $q=5$. In Taha's paper \cite{T19}, he defines an analogue of the continued fraction algorithm for each Hecke group, similarly extending what we do here in Definition \ref{def:algorithm} and Proposition \ref{prop:terminates} for the case $q=5$. Regarding Hecke groups and their finite-index subgroups, see also the paper by Cheng Lien Lang and Mong Lung Lang \cite{lang}.

\subsection*{Acknowledgements}
This project began at the Oberwolfach workshop ``Flat Surfaces and Dynamics on Moduli Space'' in 2014. We are grateful to ``Dynamics and Geometry in the Teichm\"uller Space'' at CIRM in 2015, ``Flat Surfaces and Dynamics of Moduli Space'' at BIRS in Oaxaca in 2016, ``Teichm\"uller Space, Polygonal Billiard, Interval Exchanges'' at CIRM in 2017, the Billiards Cluster at Tufts University in 2017, ``Teichm\"uller Dynamics, Mapping Class Groups and Applications'' at Institut Fourier in 2018, DD's visits to the IHES in spring 2020 and 2022, and the Summer@ICERM program in 2021, which allowed for our continued collaborations.

All of our computations were done in \verb|Sage|, many using \verb|CoCalc| \cite{sage, cocalc}. We thank Pierre Arnoux, Pascal Hubert, Jordan Emme, Sergei Tabachnikov, Ronen Mukamel, Barak Weiss, Corinna Ulcigrai, Curt McMullen and John Smillie for productive conversations about this project.

\section{The double pentagon and the golden L}

A classical construction attributed to Zemylakov and Katok \cite{fox,katok} relates billiards on a rational polygon to a linear flow on a translation surface, by reflecting the billiard table across its edges until each edge is matched with an opposite parallel edge of the same length. It is easier to study linear flows than billiard paths, because the direction of the flow does not change.

We first unfold the pentagon to its corresponding flat surface, called the \emph{necklace}, which is tiled by 10 copies of the pentagon. Veech gave a presentation of the necklace in \cite[Figure 2.3]{veech}. The necklace is a 5-fold cover of the \emph{double pentagon}. For previous work on the double pentagon, see \cite{davis13}. Since the double pentagon is simpler, we consider trajectories there instead of on the necklace. Then we perform an affine automorphism of the double pentagon to transform it into the golden L. The golden L is easier to work with than the double pentagon, because it is a right-angled surface.

\subsection{From the pentagon to the necklace, double pentagon and golden L}
\label{subsec:necklace}

We unfold the pentagon billiard table across its edges, labeling each image edge with the same label as the original billiard table edge, until each edge has an oppositely-oriented, parallel edge with the same label. This requires 10 copies of the pentagon (the order of the dihedral group $D_5$), which can be arranged in a circle or \emph{necklace}, shown in the middle of Figure \ref{fig:necklace}. The edge identifications, which are inherited from the unfolding of the pentagon table, are as labeled in the figure.

\begin{figure}[!ht]
  \includegraphics[width=\textwidth]{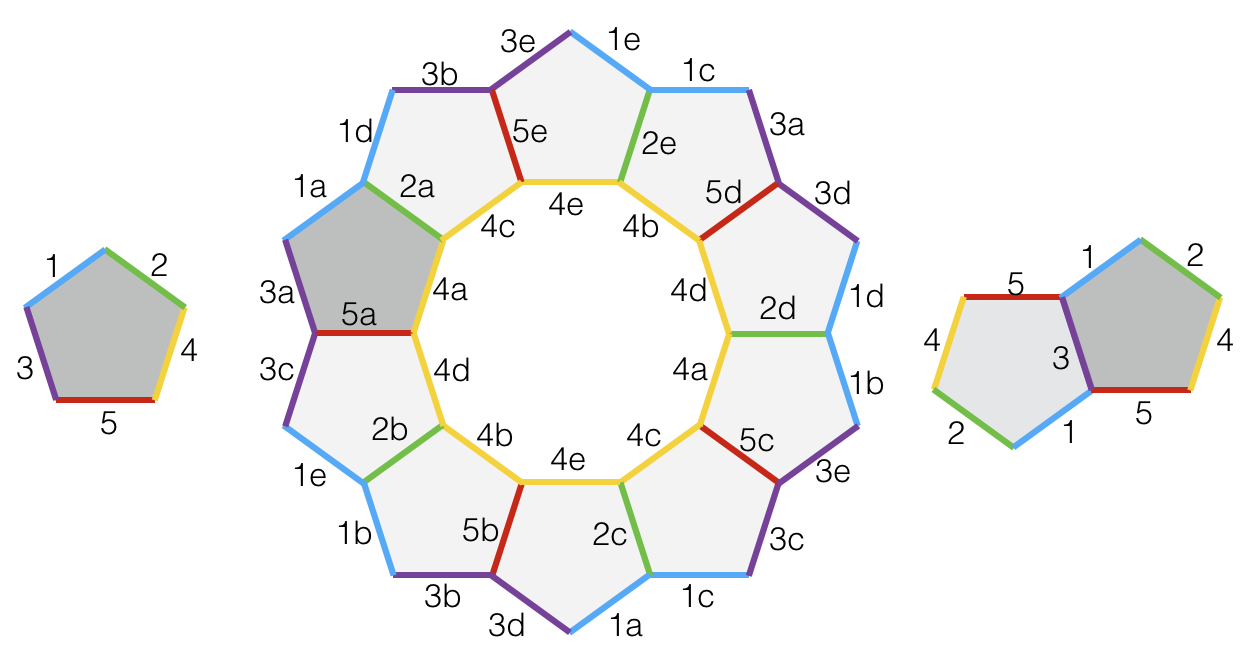}
  \caption{Left: the regular pentagon billiard table. Center: the unfolding of the pentagon billiard table into the necklace translation surface. The number 1-5 (along with the color) indicates which edge of the billiard table was unfolded to get that edge, and the letter indicates which of the copies are glued to each other. Right: The double pentagon translation surface.  \label{fig:necklace}}
\end{figure}

\begin{definition}\label{def:goldenL}
  The golden ratio $\phi = \frac {1+\sqrt{5}}2 \approx 1.618\ldots$ is the positive solution of $x^2=x+1$.  Its inverse $1/\phi = \phi - 1$ satisfies the equation $x^2=1-x$.

  The \emph{golden L} is an L-shaped translation surface with length-to-thickness ratio the golden ratio. It is built out of a $1\times 1$ square with $1\times 1/\phi$ and $1/\phi\times 1$ rectangles glued to it along their length-$1$ sides (Figure \ref{fig:golden-L-color}, and also Figures \ref{fig:cutandpaste} and \ref{fig:sector-vectors}).
  The \emph{double pentagon} is made of two regular pentagons, one of which is a reflected copy of the other, with opposite parallel edges identified.

  A \emph{cylinder} is a maximal parallel family of (primitive) periodic trajectories. The golden L and double pentagon have two cylinders in each periodic direction: a \emph{long cylinder} and a \emph{short cylinder}. In Figure \ref{fig:cutandpaste}, the long and short cylinders in the horizontal direction are shaded dark and light grey, respectively. The two cylinders have the same modulus (aspect ratio) and have length ratio $\phi$.

\end{definition}

%
%




We now turn to the double pentagon, and start by giving the affine equivalence with the golden L.

\begin{definition}\label{def:stretch}
  Let $P = \bm 1 {\cos(\pi/5)} 0 {\sin(\pi/5)}$, so $P^{-1} = \bm 1 {\cos(\pi/5)} 0 {\sin(\pi/5)}^{-1}$.
\end{definition}

\begin{lemma}
  The matrix $P$ takes the golden L surface to the double pentagon surface, and its inverse $P^{-1}$ takes the double pentagon to the golden L. In particular, they take long cylinder vectors to long cylinder vectors, and the same for short cylinder vectors.
\end{lemma}

\begin{proof}
  We can cut up the golden L on the left side of Figure \ref{fig:cutandpaste}, along the dark lines shown, and reassemble it by translation, respecting edge identifications, as shown with the dashed arrows, into the sheared version of the double pentagon outlined in the same picture. Then we can shear the double pentagon to the right, obtaining the double regular pentagon on the right side of Figure \ref{fig:cutandpaste}.

  We construct the shearing matrix that takes the golden L to the double pentagon, column by column. The vectors $[1,0]$ and $[0,1]$ are shown as thick arrows in the golden L on the left side of Figure \ref{fig:cutandpaste}. The unit horizontal edge $[1,0]$ stays as a unit horizontal. The vertical edge is mapped to a diagonal of the pentagon, also of unit length, and making an angle of $\pi/5$ with the positive horizontal.
\end{proof}

\begin{remark}
The \emph{Bouw-Möller surfaces} $S_{m,n}$ are an infinite family of Veech surfaces made by identifying edges of $m$ semi-regular $2n$-gons \cite{Hooper,DPU}. Here a \emph{semi-regular polygon} is an equiangular polygon whose edge lengths alternate between two values (possibly 0). Both surfaces described here are Bouw-Möller surfaces: the golden L is $S_{5,2}$ and the double pentagon is $S_{2,5}$.
\end{remark}

\begin{corollary}\label{cor:transform-periodicity}
  A direction $\vv$ is periodic on the double pentagon if and only if the direction $P^{-1}\vv$ is periodic on the golden L.
\end{corollary}

\begin{proof}
  The saddle connection in direction $\vv$ on the double pentagon can be cut-and-pasted into a saddle connection in direction $P^{-1}\vv$ on the golden L.
\end{proof}

\begin{figure}[!ht]
  \centering
  \label{L-to-pent}
  \includegraphics[width=0.7\textwidth]{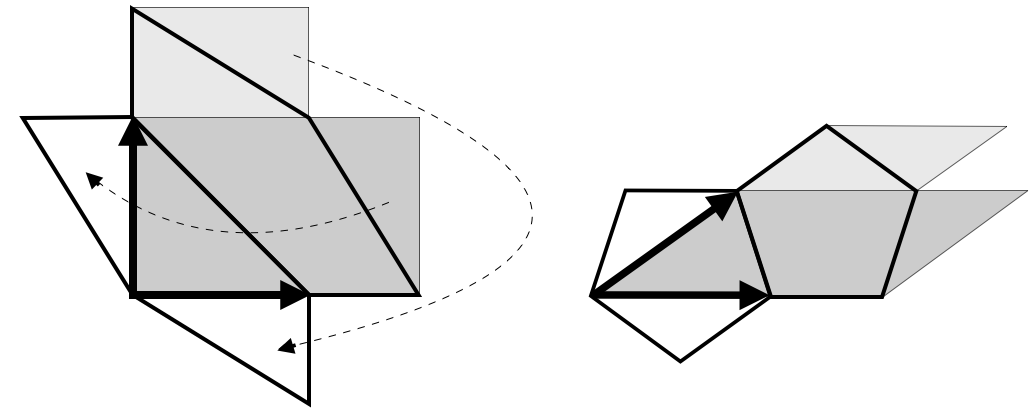}
  \caption{The cut-and-paste, shear, and vertical compression between the golden L and the double pentagon. For an earlier paper with a similar figure, see \cite[Figure 4]{curt}.  \label{fig:cutandpaste}}
\end{figure}

We are interested in linear flows on these surfaces, especially periodic trajectories. A linear flow is defined by a nonzero vector. By the $4$-fold rotational symmetry of the golden L, we may reduce our attention to flows whose direction $\theta$ satisfies $\theta\in[0,\pi/2)$. By the $10$-fold rotational symmetry of the double pentagon, we may reduce our attention to flows whose direction $\theta$ satisfies $\theta\in[0,\pi/5)$. The linear map $P$ between the golden L and the double pentagon takes these sectors to each other.

With this correspondence, we take advantage of the best of both worlds: we consider trajectories on both sides, thinking of their vectors on the golden L side (where vectors are easier), and thinking of their cutting sequences on the double pentagon side (where cutting sequences are nice).

\subsection{The tree of saddle connection vectors for the golden L}\label{sec:tree}

%
%
%
%
%
%

As stated in the introduction, we wish to study the directions of periodic trajectories on the golden L. Periodic directions on the golden L are all of those with slope in $\mathbf{Q}[\sqrt 5]$. This is quite special; for example, periodic directions on the double regular heptagon are a proper subset of the field that they generate  \cite{arnoux}.

Periodic trajectories come in cylinders (Definition \ref{def:goldenL}), and have a naturally associated vector, the \emph{holonomy vector} of this cylinder, obtained by developing the corresponding primitive closed geodesic in the plane; likewise for saddle connections. In each direction of a periodic trajectory on the golden L, there are two cylinders of periodic trajectories, one short and one long, with length ratio $\phi$ (shown in white and black in Figure \ref{fig:saddle-connections}). The horizontal and vertical directions are examples of this. Examples for pentagon billiard trajectories in each of the two cylinders are in Figure \ref{fig:buddies}.

\begin{figure}[!ht]
  \centering
  \includegraphics[width=0.3\textwidth]{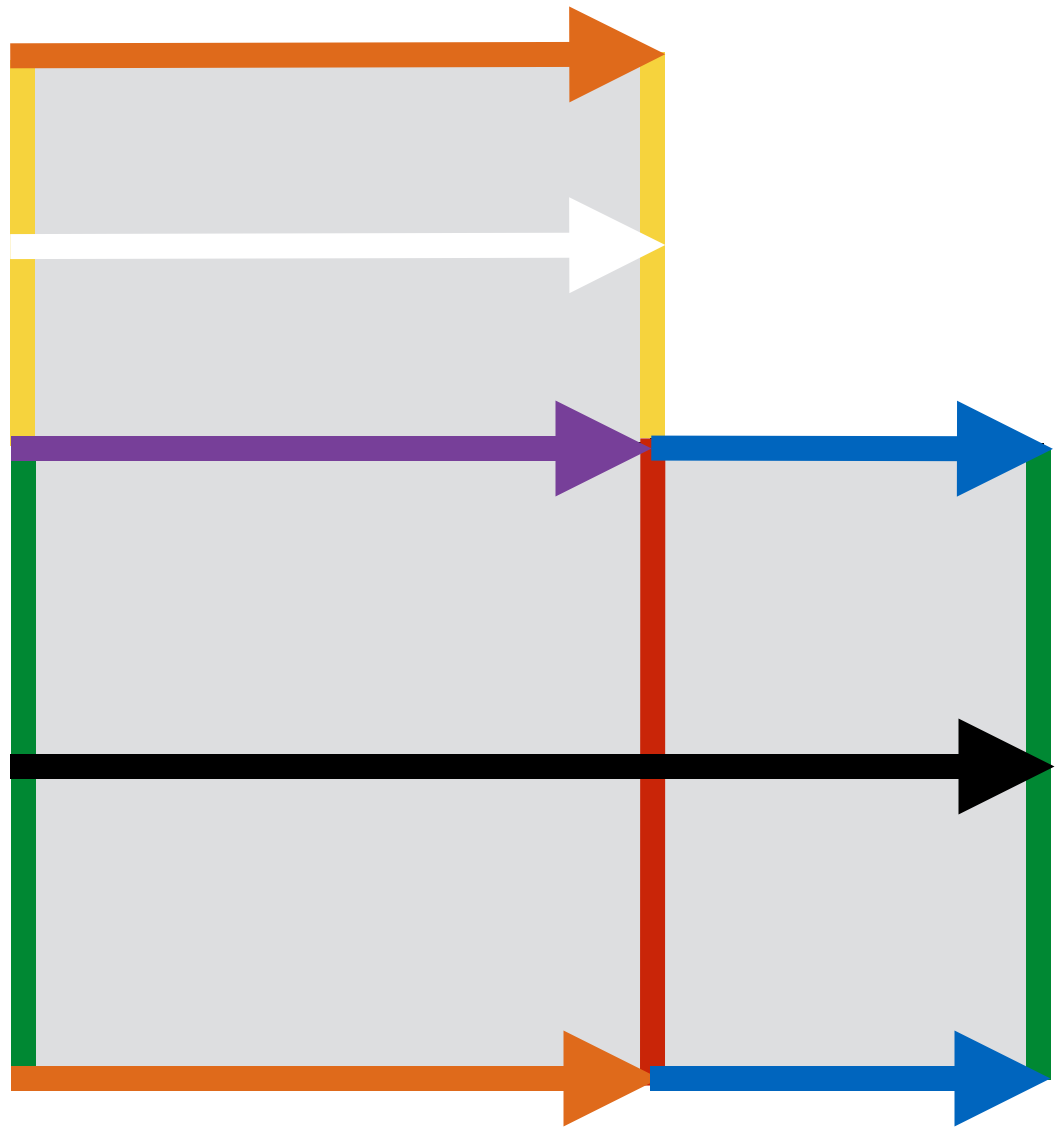}
  \caption{The golden L showing saddle connection and cylinder vectors in the horizontal direction: the short saddle connection vector (blue) and the two long saddle connection vectors (orange and purple), and the short cylinder vector (white) and long cylinder vector (black). Notice that the short cylinder vector coincides with the long saddle connection vector, and the long cylinder vector is the sum of the short and long saddle connection vectors.  \label{fig:saddle-connections}}
\end{figure}

\begin{figure}[!ht]
  \centering
  \includegraphics[width=0.23\textwidth]{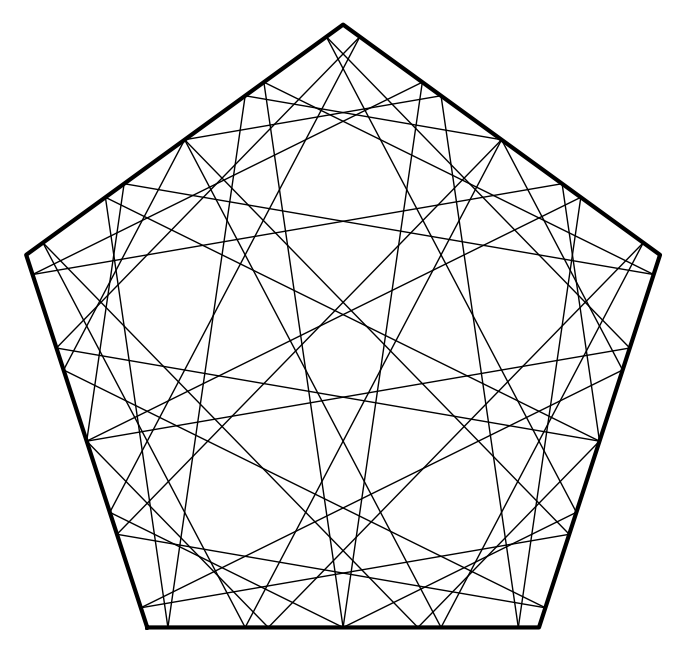}
  \includegraphics[width=0.23\textwidth]{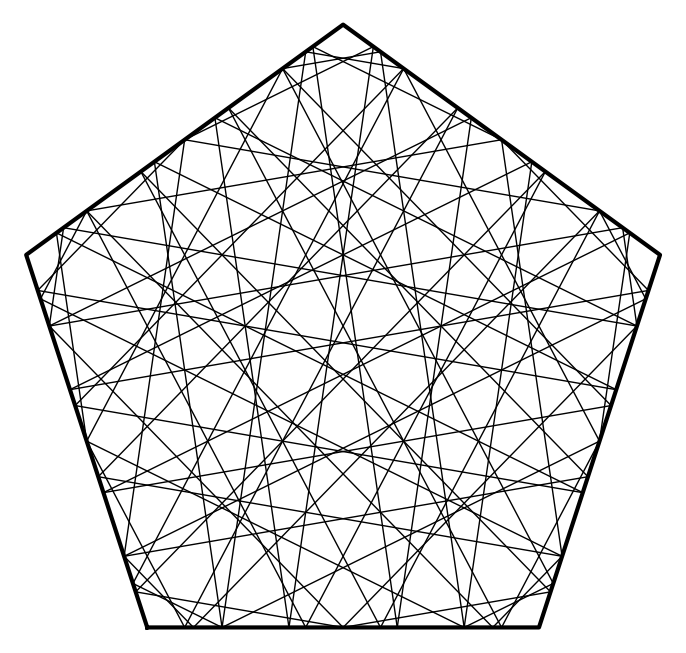} \ \ \ \ \ \ \ \
  \includegraphics[width=0.23\textwidth]{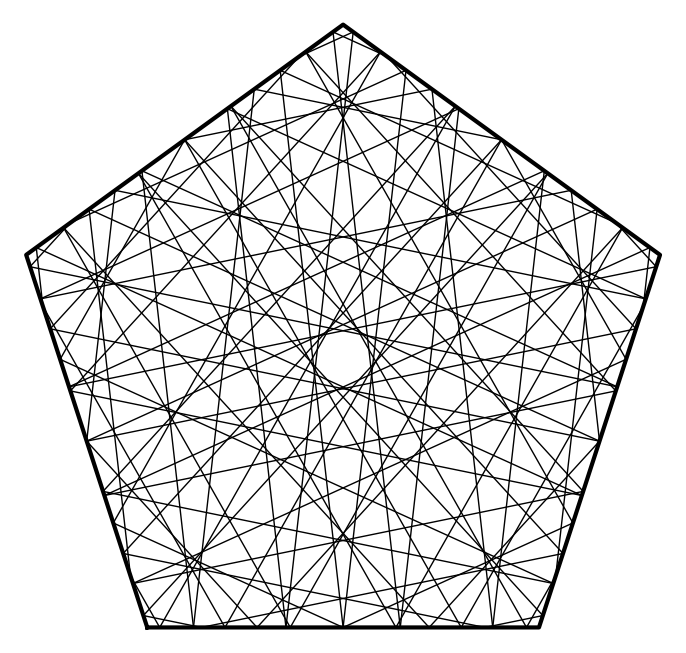}
  \includegraphics[width=0.23\textwidth]{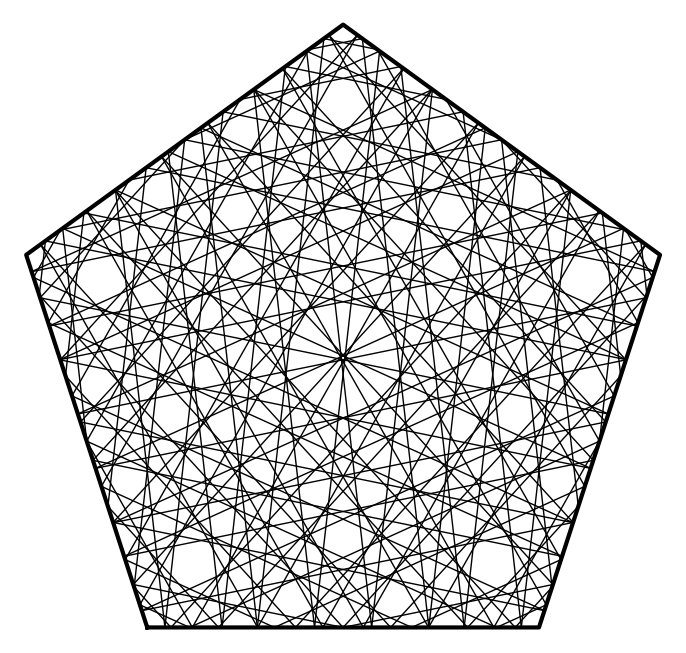}
  \caption{Two pairs of short and long billiard trajectories. For each pair, the trajectory is in the same direction. Notice that within each pair, the ``empty spots'' in one tend to correspond to ``dense spots'' in the other. This is because at an empty spot in the left pentagon, no core curve of a short cylinder crosses there, so the core curves of the long cylinder are more likely to cross that area in the right pentagon, and vice-versa. These have tree words 23 and 130, respectively.
      \label{fig:buddies}}
\end{figure}

In each periodic direction, there are also two lengths of saddle connection, with length ratio $\phi$. The long saddle connection vector coincides with the short cylinder vector, and the sum of the long and short saddle connection vectors coincides with the long cylinder vector (Figure \ref{fig:saddle-connections}). Again, the horizontal and vertical directions are examples of this.

Our approach to enumerating periodic directions is to take advantage of the affine self-similarities of the golden L (Veech's ``hidden symmetries''). We single out four such affine symmetries in the following definition.

\begin{definition}\label{def:sigmas}
  Let $\Sigma = \left\{\sv xy : x>0,y\geq 0\right\}$, the positive cone (first quadrant).
  For $i=0,1,2,3$, define $\Sigma_i$ and $\sigma_i$, as follows. 

  $\sigma_0 = \sm 1 \phi 0 1$, and $\Sigma_0 = \sigma_0 \Sigma = \left\{\sv xy : 0 \leq y < x/\phi\right\}$.

  $\sigma_1 = \sm \phi \phi 1 \phi$, and $\Sigma_1 = \sigma_1 \Sigma = \left\{ \sv xy : x/\phi \leq y < x \right\}$.

  $\sigma_2 = \sm \phi 1 \phi \phi$, and $\Sigma_2 = \sigma_2 \Sigma = \left\{\sv xy : x \leq y < \phi x\right\}$.

  $\sigma_3 = \sm 1 0 \phi 1$, and $\Sigma_3 = \sigma_3 \Sigma = \left\{\sv xy : \phi x \leq y \right\}$.
\end{definition}

\begin{lemma} We have the following facts about the $\sigma_i$ and $\Sigma_i$:
  \begin{enumerate}
    \item The first quadrant $\Sigma$ is partitioned into $\Sigma_0$, $\Sigma_1$, $\Sigma_2$, $\Sigma_3$.
    \item The matrices $\sigma_i$ generate the Veech group.
  \end{enumerate}
\end{lemma}

\begin{proof}
  The first claim is clear by construction. For the second, notice that $\sigma_0 \cdot \sigma_3^{-1} \cdot \sigma_0 \cdot \sigma_3^{-1} \cdot \sigma_0 = \sm 0{-1}10$, the matrix for a quarter-turn rotation, which we can call $\rho$. Together $\rho$ and $\sigma_0$ generate the Veech group for the golden L, the $(2,5,\infty)$ triangle group, a maximal discrete subgroup of $SL_2(\R)$, which therefore must be the Veech group of the golden L.
\end{proof}

\begin{figure}[!ht]
  \begin{center}
    \includegraphics[width=100pt]{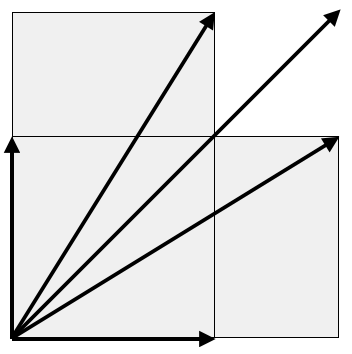}
    \caption{The golden L, and the vectors that are the column vectors for the matrices $\sigma_i$ in Definition \ref{def:sigmas}, which bound the associated sectors $\Sigma_i$ in the first quadrant. The slopes of the middle three dividing lines are $1/\phi$, $1$, and $\phi$. \label{fig:sector-vectors}}
  \end{center}
\end{figure}

We will use these matrices to generate all of the saddle connection vectors.

\begin{remark}
  The matrices $\sigma_0$ and $\sigma_3$ alone generate the Veech group, but we choose to use all four $\sigma_i$ because they give a partition of the first quadrant into four sectors.
\end{remark}

\begin{definition}
  Let $\Lambda$ (respectively $\Lambda_i$) be the intersection of the set of \emph{short cylinder vectors} (equivalently, the set of long saddle connection vectors) with $\Sigma$ (respectively $\Sigma_i$).
\end{definition}

\begin{lemma}
  The $\Lambda_i$s form a partition of $\Lambda$, and $\Lambda_i = \sigma_i \Lambda$.
\end{lemma}

\begin{proof}
  The first statement is clear by construction.

  We have $\sigma_i \Lambda \subseteq \Lambda \cap \Sigma_i$ because $\Lambda \subseteq \Sigma$, so the images satisfy $\sigma_i \Lambda \subseteq \sigma_i \Sigma = \Sigma_i$. Since long saddle connection vectors go to long saddle connection vectors, for any $v\in\Lambda$, $\sigma_i v$ is a long saddle connection vector that lies in $\Sigma$, so it is also in $\Lambda$.

  For the reverse inclusion $\sigma_i \Lambda \supseteq \Lambda \cap \Sigma_i$, let $v\in \Lambda \cap \Sigma_i$. Then $\sigma_i^{-1} v$ is also a long saddle connection vector, and it is in $\Sigma$, so it is in $\Lambda$, and so $v = \sigma_i(\sigma_i^{-1})v \in \Lambda_i$.
\end{proof}

\begin{definition}
  Let $A = \{\sigma_0,\sigma_1,\sigma_2,\sigma_3\}$, and let $A^*$ be the set of all finite words on the alphabet $A$, including the empty word. For an element $\sigma = \sigma_{k_1}, \sigma_{k_2}, \ldots, \sigma_{k_n}$ of $A^*$, let $m(\sigma) = \sigma_{k_n}\cdots\sigma_{k_2}\sigma_{k_1}$ be the product of the matrices comprising $\sigma$, in order of application.
\end{definition}

\begin{theorem}[Tree Theorem]\label{thm:saddles}
  The set of all vectors $m(\sigma)\cdot \sv 10$, for $\sigma\in A^*$, gives the entire set $\Lambda$.
  Furthermore, there is a bijection between $\Lambda$ and the set of words not starting with $\sigma_0$.
\end{theorem}

\begin{proof}
  Notice that, for vectors in $\Sigma$, all of the $\sigma_i$s send them to longer vectors (except $\sv 10$, which is sent to itself). Thus the $\sigma_i^{-1}$s send vectors in $\Lambda_i$ to shorter vectors in $\Lambda$ (except $\sv 10$, which is sent to itself). Since $\Lambda$ is discrete, the image eventually lands on the shortest vector, $\sv 10$. Working backwards, starting with $\sv 10$ and applying the $\sigma_i$s, gives the result. Since $\sigma_0 \sv 10 = \sv 10$, and this is the only time when $\sigma_i \vv = \vv$ for $\vv\in\Lambda$, we omit words starting with $\sigma_0$ and get the desired bijection.
\end{proof}

Theorem~\ref{thm:saddles} gives a way of generating all saddle connection vectors on the golden L (Figure \ref{fig:saddle_connections}). For a different presentation of the same construction, see \cite[\S2.2 and Theorem 4]{DFT}.

\begin{figure}[!ht]
  \begin{center}
    \includegraphics[width=0.48\textwidth]{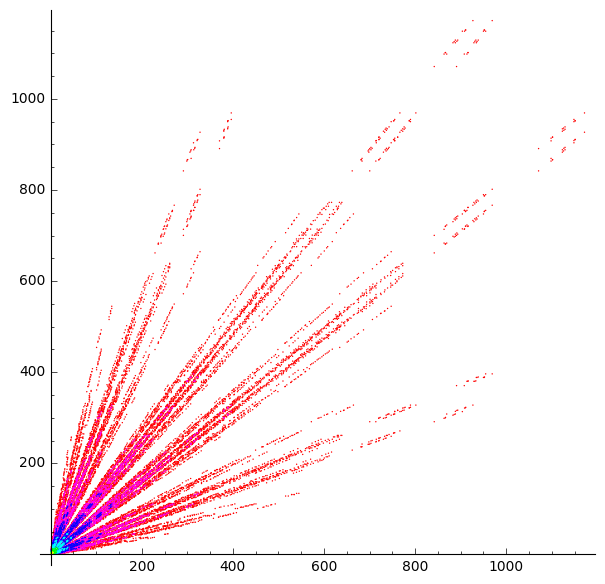} \ \
    \includegraphics[width=0.48\textwidth]{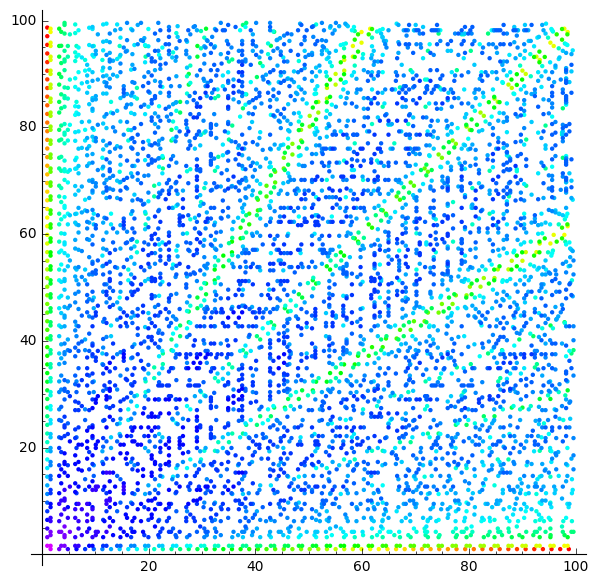} \ \
    \caption{(a) Golden L long saddle connection vectors from the first six levels of the tree (3075 points). (b) All of the long saddle connection vectors in $[0,100]^2$, which requires 61 levels of the tree; see Corollary~\ref{cor:depth_needed} (6007 points). Points are color coded by level; the color scale is different in the two pictures. \label{fig:saddle_connections}}
  \end{center}
\end{figure}

\newpage
\begin{proposition}
  At depth $k\geq 1$ in the tree, the shortest vectors are $\sv {k\phi}1$ and $\sv 1{k\phi}$.
\end{proposition}

\begin{proof} Beginning with $\sv 10$, applying the matrices $\sigma_1, \sigma_2, \sigma_3$ yields the vectors $\sv \phi 1$, $\sv 1\phi$, $\sv \phi\phi$, respectively. The latter is longer than the first two, so we discard it.
  Then the result follows from the following claims for any $\vv = \sv ab$ in the positive cone:
  \begin{enumerate}
    \item $|\sigma_0 \vv| < |\sigma_1 \vv|$ and $|\sigma_3 \vv| < |\sigma_2 \vv|$.
    \item $|\sigma_0 \sv ab| < |\sigma_3 \sv ab|$ if $a>b$ and $|\sigma_0 \sv ab| > |\sigma_3 \sv ab|$ if $a<b$.
  \end{enumerate}
  Both of these follow from multiplying out the matrices and considering the lengths of the resulting vectors. By (1), the shortest vector will always be a product of $\sigma_0$ and $\sigma_3$ applied to $\vv$. By (2), applying ${\sigma_0}^k$ or ${\sigma_3}^k$ results in a shorter vector than a product including both matrices, and which one of these is shorter depends on the location of the vector.
\end{proof}

\begin{corollary}\label{cor:depth_needed}
  To obtain all of the vectors in the square $[0,N]^2$, it is necessary to go to depth $N/\phi$ in the tree.
\end{corollary}

For example, in Figure~\ref{fig:saddle_connections}, depth $61 =\lfloor 100/\phi\rfloor$ is required to get everything in $[0,100]^2$.

\begin{definition}\label{def:tree}
  We refer to the \emph{tree of periodic directions} as the tree given by the construction in Theorem \ref{thm:saddles}: The base node of the tree is vector $\sv 10$, the first level is given by $\sigma_i \sv 10$ for $i\in\{1,2,3\}$, and the $(n+1)^\text{st}$ level for $n\geq 1$ is given by $\sigma_{j_n}\cdots\sigma_{j_1} \sigma_i \sv 10$ for $i\in\{1,2,3\}, j_k\in\{0,1,2,3\}$. The \emph{tree word} corresponding to a periodic direction is given by the finite sequence of subscripts of the transformations $\sigma_i$ in the order of application, e.g. tree word 120 for the direction $\sigma_0 \sigma_2 \sigma_1 \sv 10$.
\end{definition}

\begin{figure}[!ht]
  \begin{center}
    \begin{tikzpicture}[
        scale = 0.98, transform shape, thick,
        every node/.style = {},
        grow = down,  
        level 1/.style = {sibling distance=5cm},
        level 2/.style = {sibling distance=1.3cm},
        level 3/.style = {sibling distance=1cm},
        level 4/.style = {sibling distance=1cm},
        level distance = 3cm
      ]
      \node  (Start) { \includegraphics[width=0.08\textwidth]{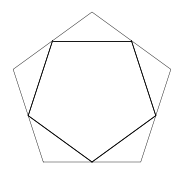}}
      child {   node (A) {\includegraphics[width=0.08\textwidth]{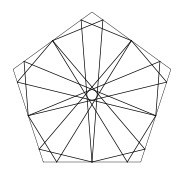}}
          child { node (AB) {\includegraphics[width=0.08\textwidth]{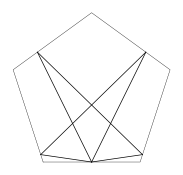}}
            }
          child { node (AC) {\includegraphics[width=0.08\textwidth]{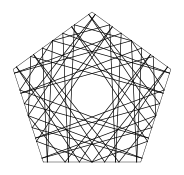}}}
          child { node (AD) {\includegraphics[width=0.08\textwidth]{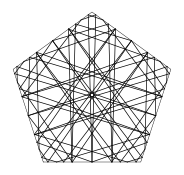}}}
          child { node (AE) {\includegraphics[width=0.08\textwidth]{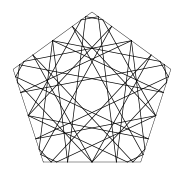}}}
        }
      child {   node  (D) {\includegraphics[width=0.08\textwidth]{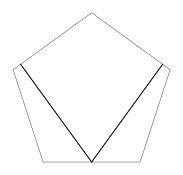}}
          child { node  (DB) {\includegraphics[width=0.08\textwidth]{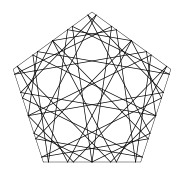}}}
          child { node  (DC) {\includegraphics[width=0.08\textwidth]{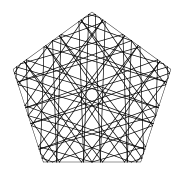}}}
          child { node  (DE) {\includegraphics[width=0.08\textwidth]{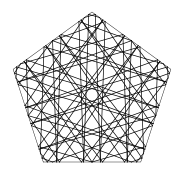}}}
          child { node  (DF) {\includegraphics[width=0.08\textwidth]{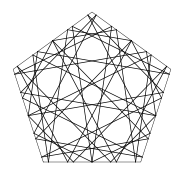}}}
        }
      child {   node  (G) {\includegraphics[width=0.08\textwidth]{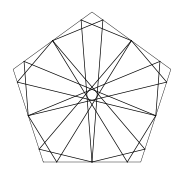}}
          child { node  (GB) {\includegraphics[width=0.08\textwidth]{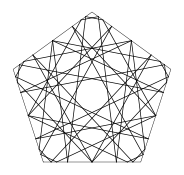}}}
          child { node  (GC) {\includegraphics[width=0.08\textwidth]{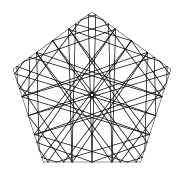}}}
          child { node  (GD) {\includegraphics[width=0.08\textwidth]{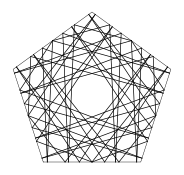}}}
          child { node  (GE) {\includegraphics[width=0.08\textwidth]{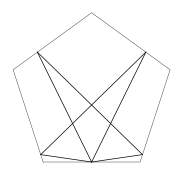}}}
        }  ;

      \begin{scope}[nodes = {draw = none}]
        \path (Start) -- (A) node [near start, right]  {$1$};
        \path (Start) -- (D) node [near start, left] {$2$};
        \path (Start) -- (G) node [near start, left] {$3$};

        \path (A)     -- (AB) node [near end, right]  {$0$};
        \path (A)     -- (AC) node [near end, right] {$1$};
        \path (A)     -- (AD) node [near end, left]  {$2$};
        \path (A)     -- (AE) node [near end, left] {$3$};

        \path (D)     -- (DB) node [near end, right]  {$0$};
        \path (D)     -- (DC) node [near end, right] {$1$};
        \path (D)     -- (DE) node [near end, left]  {$2$};
        \path (D)     -- (DF) node [near end, left] {$3$};

        \path (G)     -- (GB) node [near end, right] {$0$};
        \path (G)     -- (GC) node [near end, right] {$1$};
        \path (G)     -- (GD) node [near end, left] {$2$};
        \path (G)     -- (GE) node [near end, left] {$3$};

        \begin{scope}[nodes = {below = 15pt}]
          \node at (AB) {$10$};
          \node at (AC) {$11$};
          \node at (AD) {$12$};
          \node at (AE) {$13$};
          \node at (DB) {$20$};
          \node at (DC) {$21$};
          \node at (DE) {$22$};
          \node at (DF) {$23$};
          \node at (GB) {$30$};
          \node at (GC) {$31$};
          \node at (GD) {$32$};
          \node at (GE) {$33$};
        \end{scope}
      \end{scope}
    \end{tikzpicture}
    \caption{The first two levels in the tree of periodic directions on the regular pentagon. The short trajectory corresponding to the core curve of the cylinder is shown for each. \label{fig:tree}}
  \end{center}
\end{figure}

The short pentagon billiard trajectory for each direction in the tree is in Figure~\ref{fig:tree}. The picture hints that the tree is left-right symmetric:

\begin{theorem}\label{thm:tree_symm}
  The tree is symmetric:  to get an identical trajectory, for the first digit of the tree word we do the swap $1\leftrightarrow 3$, and for all the rest of the digits, we do the swap $0\leftrightarrow 3$ and $1\leftrightarrow 2$.
\end{theorem}

We need the following results, which follow from matrix multiplication.

\begin{lemma}\label{lem:transpose}
  If $A = \sm abca$, $B = A^T$, and $S = \sm 0110$, then
  \begin{enumerate}
    \item $SAS = B$ and $SBS = A$ and $AS = SB$ and $SA = BS$.
    \item $A \sv xy = S B \sv yx$.
  \end{enumerate}
\end{lemma}

\begin{proof}[Proof of Theorem \ref{thm:tree_symm}]
  Observe that $v_2 := \sigma_2 \sv 10$ satisfies $S v_2 = v_2$, and $v_1 := \sigma_1 \sv 10$ and $v_3 := \sigma_3 \sv 10$ satisfy $S v_1 = v_3$ and $S v_3 = v_1$. This proves the result for words of length 1.
  Since we can apply Lemma \ref{lem:transpose} to $(A, B) = (\sigma_0, \sigma_3)$ and to $(A, B) = (\sigma_1, \sigma_2)$, the result follows by induction.

  The above shows that the directions on the golden L corresponding to a tree word and its swapped word are symmetric across the line $y=x$. Applying the matrix $P$ to directions that are symmetric on the golden L results in directions that are symmetric across the long diagonal of the double pentagon, so they stay symmetric when lifted to the necklace, and when folded into a pentagon billiard trajectory.
\end{proof}

\begin{remark}
  On the golden L or on the double pentagon, the central representative of a cylinder of periodic trajectories passes through two regular Weierstrass points. In the double pentagon, these are the midpoints of the edges; in the golden L these are the centers of the rectangles and square, and the midpoints of the short sides of the rectangles. After lifting to the necklace and folding to the pentagon billiard table, both points are mapped to the midpoint of the same edge, or of all of the edges in the case of trajectories with rotational symmetry. One can observe in the pictures that anytime a trajectory hits the midpoint of a given edge, it does so twice. Also see a related discussion in \cite[\S 5]{dodecahedron}.
\end{remark}

\subsection{Directions to saddle connection vectors on the golden L}\label{sec:cont_frac}

Theorem \ref{thm:saddles} gives us all of the saddle connection vectors for the golden L. In this section, we answer a related question: \emph{Given a periodic direction on the golden L, what is the saddle connection vector in that direction?}

\begin{remark}
  Given a rational slope $a/b$ in lowest terms on the square torus, we can easily see that the associated saddle connection vector in that direction is $[b,a]$. However, suppose that we are given a direction such as $\left[1618/3141, 7776/72785\right]$ on the square torus. It is clear that this direction is periodic, because it corresponds to a rational slope. However, in order to get the slope in lowest terms, we need to perform the Euclidean algorithm on the two fractions. What follows is an analogous process for periodic directions on the golden L.
  %
  %
\end{remark}

\begin{definition} \label{def:algorithm}
  We define an iterated process on a direction vector $\vv$ in the first quadrant:

  Let $\vv_0 = \vv$. For $i\ge 0$, let $\vv_{i+1} = \sigma_{k_i}^{-1} \vv_{i}$ where $\vv_{i} \in \Sigma_{k_i}$. Let $a_\vv = \{k_0, k_1, \ldots\}$.

\end{definition}

Here $\vv_0$ is the original direction vector, which is in some sector $\Sigma_{k_0}$. The operation $\sigma^{-1}_{k_0}$ stretches sector $\Sigma_{k_0}$ into the entire first quadrant, and $\sigma^{-1}_{k_0} \vv_0$ is the image vector of $\vv_0$ under this operation. We iterate this process, repeatedly stretching whichever sector our vector lies in, into the entire first quadrant. The sequence $a_\vv$ is the \emph{itinerary} of sectors that the vector lies in at each step of the process. This is analogous to the continued fraction algorithm. If our vector is in a periodic direction, eventually one of its images will land on the boundary of the sector, so the process stops:

\begin{proposition}\label{prop:terminates}
  If $\vv$ is in a saddle connection direction, then
  \begin{enumerate}
    \item The sequence $a_\vv$ is eventually constant equal to $0$, starting
          from some index $k$, and  \label{part:terminates}
    \item For this value of $k$, $\vv_k= \sv {\ell_\vv} 0$ for some ${\ell_\vv} \in \R$. \label{whatisell}
  \end{enumerate}
\end{proposition}

\begin{proof}
  By Theorem \ref{thm:saddles}, every saddle connection vector can be obtained by applying a finite number of $\sigma_i$'s to $\sv 10$. Thus, a finite number of repeated applications of the appropriate $\sigma_i^{-1}$ will transform any saddle connection vector to $\sv 10$. Since $\vv$ is in a saddle connection direction, it is a constant multiple of a saddle connection vector, so a finite number of repeated applications of the appropriate $\sigma_i^{-1}$'s will transform it to a constant multiple of $\sv 10$.
  After this finite number of applications, the transformed vector is ${\ell_\vv} \sv 10$ in $\Sigma_0$, and $\sigma_0 ({\ell_\vv}\sv 10) = {\ell_\vv}\sigma_0\sv 10= {\ell_\vv}\sv 10$, so it stays in $\Sigma_0$.
\end{proof}


We can use this process to give the saddle connection vectors in a particular direction:

\begin{proposition} \label{prop:scvectors}
  Given a periodic direction $\vv$ on the golden L, the long and short saddle connection vectors in that direction are $\vv / \ell_\vv$ and $\vv / (\phi\ell_\vv)$, respectively, where $\ell_\vv$ is as in Proposition~\ref{prop:terminates}.
\end{proposition}

\begin{proof}
  The maps $\sigma_i$ are linear. By Proposition \ref{prop:terminates}(2), applying a sequence of $\sigma_i$s eventually yields a horizontal vector of the form $\sv {\ell_\vv} 0$. The horizontal long saddle connection vector is $\sv 10$, so if we end with $\sv {\ell_\vv} 0$ it means we started from $\ell_\vv$ times a long saddle connection vector. The short saddle connection vector is $1/\phi$ times the long saddle connection vector.
\end{proof}

\begin{theorem}\label{thm:abcd}
  Corresponding to any periodic direction vector $\vv$ is a unique tree word $k_0 k_1\cdots k_n$, which is an itinerary of sectors as in Definition~\ref{def:algorithm}, such that $\vv = \ell_\vv \sigma_{k_0} \ldots \sigma_{k_n} \sv{1}{0}$ for some length $\ell_\vv$. Then the saddle connection vector in the direction $\vv$ is $\sigma_{k_0}\ldots\sigma_{k_n}\sv 10$. Furthermore, the saddle connection vector is of the form $\sv{a+b\phi}{c+d\phi}$ for $a,b,c,d\in\N$.
\end{theorem}

\begin{proof}
  That the saddle connection vector is $\sigma_{k_0}\ldots\sigma_{k_n}\sv 10$ is clear by construction. Multiplying out the product of matrices $\sigma_{k_i}$ with $\sv 10$ yields a vector consisting of a sum of positive integers and positive integer multiples of powers of $\phi$. Applying the identity $\phi^2 = 1+\phi$ reduces each coordinate to the sum of a positive integer and a positive integer multiple of $\phi$. Said another way, the vector $\sv 10$ and each matrix $\sigma_i$ each have all of their entries in the ring $\Z[\phi]$, so their products do as well.
\end{proof}

\section{Combinatorial periods of pentagon and double pentagon trajectories}

\subsection{Cutting sequences on the double pentagon}\label{sec:cutting}

\begin{definition}
  Given a translation surface decomposed into polygons with labeled edges,
  and an oriented bi-infinite linear trajectory $\tau$ on this translation surface, let $c(\tau)$
  be the bi-infinite sequence of labels of the edges that the trajectory crosses.
  If the trajectory hits a corner, it stops; in this case, the sequence
  is not bi-infinite; otherwise it is.
\end{definition}

\begin{definition}\label{def:transition}
  The \emph{transition} $ab$ is \emph{allowed} in sector $\Sigma_i$ if some trajectory in $\Sigma_i$ cuts through edge $a$ and then through edge $b$.

  For $i=0,1,2,3$, the transition diagram $\D_i$ is a directed graph whose vertices are edge labels of the double pentagon (the numbers $1,\ldots,5$), with an arrow from edge label $a$ to edge label $b$ if and only if the transition $ab$ is allowed in $\Sigma_i$.
  We say that a sequence is \emph{admissible} in diagram $\D(i)$ if it is given by a path following arrows on $\D(i)$.
\end{definition}

\begin{lemma}
  For a trajectory with $\theta\in[0,\pi/5)$ on the double pentagon, the corresponding transition diagram is:

  \vspace{-1em}

  \begin{center}
    $\begin{tikzcd} 
        \1 \arrow[r, bend left=15] &
        \2\arrow[r, bend left=15] \arrow[l, bend left=15] &
        \3 \arrow[r, bend left=15]  \arrow[l, bend left=15] &
        \4 \arrow[r, bend left=15] \ \arrow[l, bend left=15] &
        \5  \arrow[l, bend left=15]  \\
      \end{tikzcd}$
  \end{center}
\end{lemma}

\vspace{-2em}

\begin{proof}
  This is easily checked by hand, using Figure \ref{four-dp}. To do this, use one of the pentagons with its edge labels (thin), and ignore the internal line segments and their bold labels.
\end{proof}

\begin{definition}
  For $i=0,1,2,3$, define $\tilde{\sigma_i} = P\sigma_i P^{-1}$.
\end{definition}

These matrices do the same job as the $\sigma_i$'s on the golden L, now on the double pentagon.

\begin{lemma}
  The image of the double pentagon under each $\tilde{\sigma_i}$ is as shown in Figure \ref{four-dp}.
\end{lemma}

\begin{figure}[!ht]
  \centering
  \includegraphics[trim={0 0.5em 0 0},clip,width=400pt]{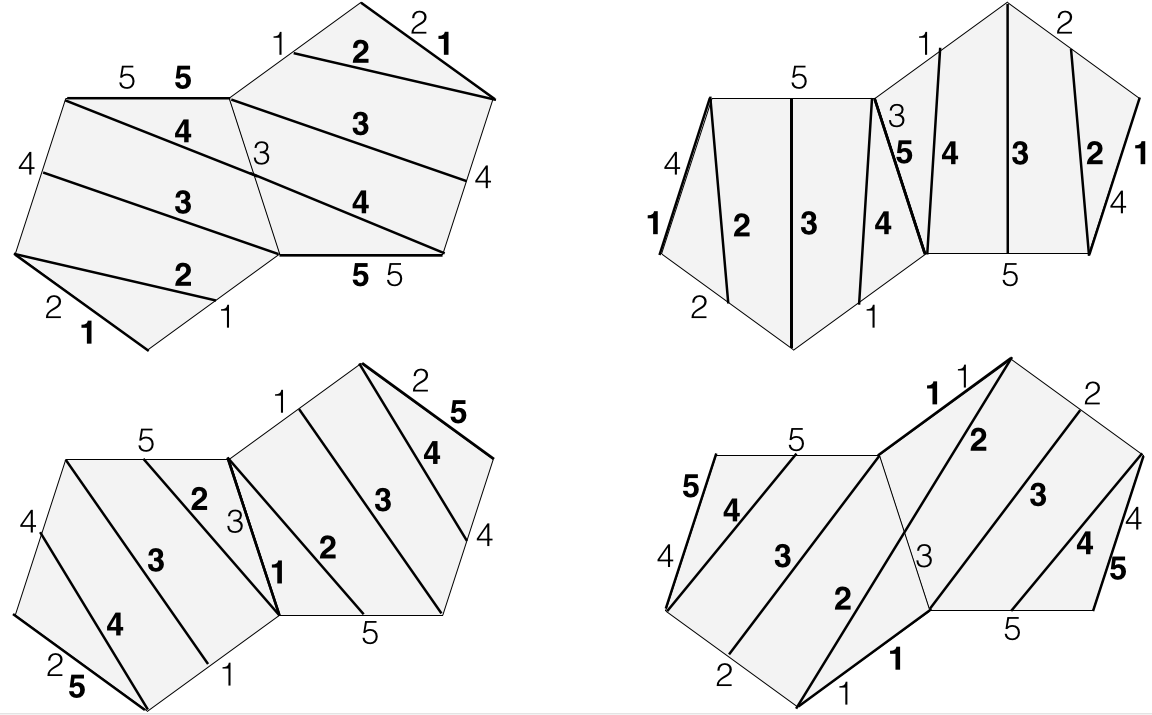}
\caption{The image of the double pentagon under each $\tilde{\sigma_i}^{-1}$: top left $\tilde{\sigma_0}^{-1}$,  bottom left $\tilde{\sigma_1}^{-1}$,  top right $\tilde{\sigma_2}^{-1}$,  bottom right $\tilde{\sigma_3}^{-1}$. \label{four-dp}}
\end{figure}

\begin{proof}
  We obtain these diagrams via direct computation in the golden L: For example,  ${\sigma_0}^{-1}$ sends $\sv 10$ to $\sv 10$ in the golden L, so it also preserves the horizontal direction in the double pentagon. ${\sigma_0}^{-1}$ sends $\sv 01$ to $\sv {-\phi} 1$ in the golden L, meaning that ${\sigma_0}^{-1}$ twists to the left by one cylinder length. The other three diagrams are obtained in the same way.
\end{proof}

\begin{lemma}\label{lem:dptd}
  For a trajectory $\tau$ with angle $0 \leq \theta < \pi/5$, if we transform $\tau$ by $\tilde{\sigma_i}$, the induced effect on the associated cutting sequence is as in the following diagrams:
\end{lemma}


$\D$: \ \  \ \  \begin{tikzcd} 
  \1 \arrow[r, bend left=15] &
  \2\arrow[r, bend left=15] \arrow[l, bend left=15] &
  \3 \arrow[r, bend left=15]  \arrow[l, bend left=15] &
  \4 \arrow[r, bend left=15] \ \arrow[l, bend left=15] &
  \5  \arrow[l, bend left=15]  \\
\end{tikzcd} \hspace{0.9in} $\1\2\3\4\5$

$\D(0)$: \begin{tikzcd} 
  \2 \arrow[r, bend left=15] &
  \1\arrow[r, bend left=15,"23"] \arrow[l, bend left=15] &
  \4 \arrow[r, bend left=15]  \arrow[l, bend left=15,"32"] &
  \3 \arrow[r, bend left=15,"4"]  \arrow[l, bend left=15] &
  \5  \arrow[l, bend left=15,"4"]
\end{tikzcd} \hspace{0.95in} $\2\1 23 \4 \3 4 \5$

$\D(1)$: \begin{tikzcd}
  \3 \arrow[r, bend left=15,"4"] &
  \5 \arrow[r, bend left=15,"432"] \arrow[l, bend left=15,"4"] &
  \1 \arrow[r, bend left=15,"23"]  \arrow[l, bend left=15,"234"] &
  \4 \arrow[r, bend left=15,"3"]  \arrow[l, bend left=15,"32"] &
  \2  \arrow[l, bend left=15,"3"]
\end{tikzcd}\hspace{1in} $\3 4 \5 432 \1 23 \4 3 \2$

$\D(2)$: \begin{tikzcd}
  \4 \arrow[r, bend left=15,"3"] &
  \2\arrow[r, bend left=15,"34"] \arrow[l, bend left=15,"3"] &
  \5 \arrow[r, bend left=15,"432"]  \arrow[l, bend left=15,"43"] &
  \1 \arrow[r, bend left=15,"2"]  \arrow[l, bend left=15,"234"] &
  \3  \arrow[l, bend left=15,"2"]
\end{tikzcd} \hspace{0.95in} $\4 3\2 34 \5 432 \1 2 \3$

$\D(3)$: \begin{tikzcd}
  \1 \arrow[r, bend left=15,"2"] &
  \3\arrow[r, bend left=15] \arrow[l, bend left=15,"2"] &
  \2 \arrow[r, bend left=15,"34"]  \arrow[l, bend left=15] &
  \5 \arrow[r, bend left=15]  \arrow[l, bend left=15,"43"] &
  \4  \arrow[l, bend left=15]  \\
\end{tikzcd} \hspace{0.95in} $\1 2 \3 \2 34 \5 \4$

On the left, we use the usual presentation of such ``generation'' transition diagrams in the literature. On the right, we use a simplified form that emphasizes the one-dimensional nature of a path on a diagram. (This was also done in \cite[\S 2.3]{DFT}.)

\begin{proof}
  Given any periodic trajectory, its cutting sequence is a periodic word in the alphabet $\{1,2,3,4,5\}$ describing the edges it crosses. We can deform this trajectory into a piecewise-linear zigzag path that connects midpoints of successive edges crossed. The original linear trajectory and the zigzag path cross the same image edges for trajectories with $\theta \in [0,\pi/5)$; this can easily be checked by hand.

  Now we associate to each node label $k$ of the transition diagram $\D(i)$ the image edge that passes through the midpoint of edge $k$, or that coincides with edge $k$. Using Figure \ref{four-dp}, we label the arrow between these image edges with the other image edges that must be crossed between original edges $k$ and $k\pm 1$, the sign depending on if they are upper arrows or lower arrows in the diagram; the edges that are crossed are also easily checked by hand.
\end{proof}

\begin{definition}\label{def:decorated}
  A \emph{decorated cutting sequence} is a cutting sequence that includes the extra information of the arrows from the transition diagram; in particular, it indicates which symbol is next. To obtain a decorated cutting sequence from a cutting sequence, we add an exponent between every pair of symbols, which is the second symbol of the pair: the periodic cutting sequence $e_1 e_2 \cdots e_k$ becomes ${e_1}^{e_2} {e_2}^{e_3} \cdots {e_k}^{e_1}$.
\end{definition}

\begin{definition}\label{def:subs}
  Define the substitution rules $r_i$ for $i=0,1,2,3$ on a decorated cutting sequence:
  \begin{equation} \label{eqn:subs}
    r_0 = \begin{cases}
      \ar 12 \to \ar 21               \\
      \ar 21 \to \ar 12               \\
      \ar 23 \to \ar 12 \ar 23 \ar 34 \\
      \ar 32 \to \ar 43 \ar 32 \ar 21 \\
      \ar 34 \to \ar 43               \\
      \ar 43 \to \ar 34               \\
      \ar 45 \to \ar 34 \ar 45        \\
      \ar 54 \to \ar 54 \ar 43
    \end{cases}
    r_1 = \begin{cases}
      \ar 12 \to \ar 34 \ar 45               \\
      \ar 21 \to \ar 54 \ar 43               \\
      \ar 23 \to \ar 54 \ar 43 \ar 32 \ar 21 \\
      \ar 32 \to \ar 12 \ar 23 \ar 34 \ar 45 \\
      \ar 34 \to \ar 12 \ar 23 \ar 34        \\
      \ar 43 \to \ar 43 \ar 32 \ar 21        \\
      \ar 45 \to \ar 43 \ar 32               \\
      \ar 54 \to \ar   23 \ar 34
    \end{cases} r_2 = \begin{cases}
      \ar 12 \to \ar 43 \ar 32                \\
      \ar 21 \to \ar  23 \ar 34               \\
      \ar 23 \to \ar 23 \ar 34 \ar 45         \\
      \ar 32 \to \ar  54 \ar 43 \ar 32        \\
      \ar 34 \to \ar  54 \ar 43 \ar 32 \ar 21 \\
      \ar 43 \to \ar 12 \ar 23 \ar 34 \ar 45  \\
      \ar 45 \to \ar 12 \ar 23                \\
      \ar 54 \to \ar 32 \ar 21
    \end{cases} r_3 = \begin{cases}
      \ar 12 \to \ar 12 \ar 23        \\
      \ar 21 \to \ar  32 \ar 21       \\
      \ar 23 \to \ar 32               \\
      \ar 32 \to \ar 23               \\
      \ar 34 \to \ar 23 \ar 34 \ar 45 \\
      \ar 43 \to \ar 54 \ar 43 \ar 32 \\
      \ar 45 \to \ar 54               \\
      \ar 54 \to \ar 45
    \end{cases}
  \end{equation}
\end{definition}

\begin{lemma}\label{lem:subs}
  Applying the substitution rules (\ref{eqn:subs}) is equivalent to implementing the transition diagrams in Lemma \ref{lem:dptd}.
\end{lemma}

In other words, applying the substitution rules (\ref{eqn:subs}) is equivalent to taking a decorated cutting sequence $c(\tau)$ on the double pentagon, where the associated cutting sequence is admissible on diagram $\D$, running the cutting sequence through diagram $\D(i)$ in Lemma \ref{lem:dptd}, collecting the letters to create a new cutting sequence, and then turning the cutting sequence into a decorated cutting sequence as described in Definition \ref{def:decorated}.

\begin{proof}
  For each diagram, we write down the node label, with the arrow label as exponent.
\end{proof}

\begin{remark}
  The literature on cutting sequences and transition diagrams contains a variety of different choices about how to symbolically represent a path on a transition diagram, and how to represent the operation that transforms one cutting sequence into another cutting sequence. John Smillie and Corinna Ulcigrai \cite{SU} used a ``dual'' transition diagram where arrow labels become nodes, and nodes become arrow labels. The first author, Irene Pasquinelli and Corinna Ulcigrai \cite{DPU} used substitutions on arrows, without using nodes or arrow labels.

  The symbol $a^b$ in the notation we use represents an arrow from $a$ to $b$, so the substitutions given in Definition~\ref{def:subs} are also substitutions on arrows. The current opinion of these and the current authors is that a direct substitution rule on arrows is the best way of expressing this type of operator. 

  Finally, we note that in the substitution rules in Definition~\ref{def:subs}, we have combined ``derivation'' and ``normalization'' (see \cite[\S 7]{DPU}) into one step. Namely, given a cutting sequence that corresponds to a path on transition diagram $\D$, running it through any of the substitutions $r_i$ yields a cutting sequence that also corresponds to a path on transition diagram $\D$, in a single step.
\end{remark}

\begin{remark}
Lemma \ref{lem:subs} combined with Theorem \ref{thm:abcd} give an algorithm for finding the cutting sequence $w$ corresponding to a given periodic direction $\vv$ on the double pentagon. The reverse operation is easier: First, let $\vv_i$ be the vector between the centers of pentagons glued along side $i$, chosen to tend in sector $(0,\pi/5)$ (Figure \ref{fig:edge_vectors}). Then we have

\begin{equation*}
\vv_1 = \bv{\cos\frac{3\pi}{10}}{\sin\frac{3\pi}{10}},\qquad
\vv_2 = \bv{\cos\frac{-3\pi}{10}}{\sin\frac{-3\pi}{10}},\qquad
\vv_3 = \bv{\cos\frac{-\pi}{10}}{\sin\frac{-\pi}{10}},\qquad
\vv_4 = \bv{\cos\frac{\pi}{10}}{\sin\frac{\pi}{10}},\qquad
\vv_5 = \bv{\cos\frac{5\pi}{10}}{\sin\frac{5\pi}{10}}.
\end{equation*}

\begin{figure}[!ht]
\centering
\includegraphics[height=120pt]{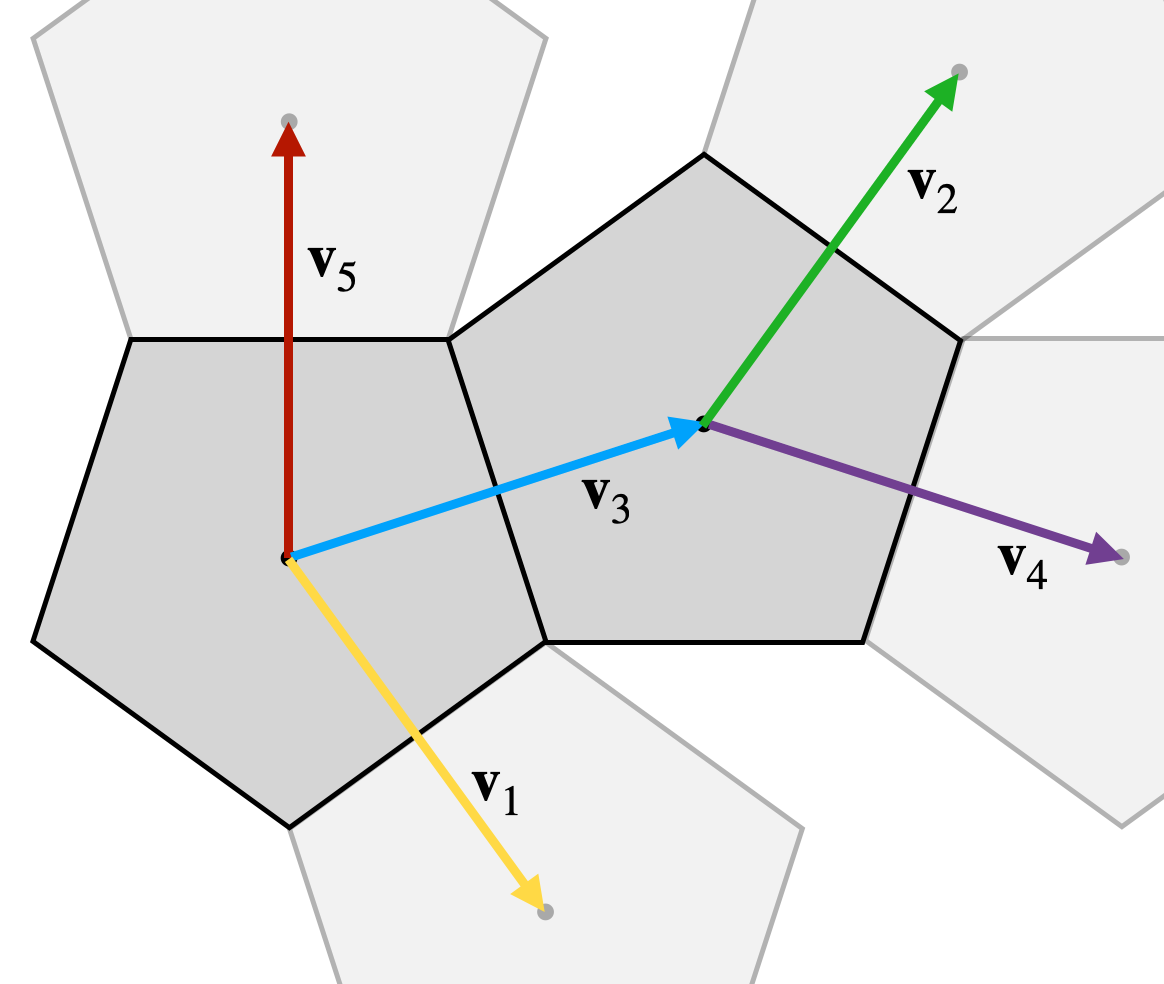}
\caption{Center-to-center vectors through each edge of the double pentagon \label{fig:edge_vectors}}
\end{figure}

Then given a cutting sequence $w \in \{1,2,3,4,5\}^*$ on the double pentagon, and letting $|w|_i$ count the number of $i$'s in $w$, we have:

\begin{equation*}
\vv =\frac{1}{\tan(\pi/5)}
|w|_1\cdot \vv_1 +
|w|_2\cdot \vv_2 +
|w|_3\cdot \vv_3+
|w|_4\cdot \vv_4+
|w|_5\cdot \vv_5.
\end{equation*}

Here $\cot(\pi/5)$ is twice the inradius of a regular pentagon with side length 1.
\end{remark}

\begin{definition}\label{def:count}
  For a given decorated periodic cutting sequence $\tau$ on the double pentagon, we count how many of the symbols are associated to each of the four pairs of arrows in the transition diagram. Specifically, define $$a_\tau = \#\{\ar 12, \ar 21\},\qquad d_\tau = \#\{\ar 23, \ar 32\}, \qquad b_\tau = \#\{\ar 34, \ar 43\}, \qquad c_\tau = \#\{\ar 45, \ar 54\}.$$
\end{definition}

The variable names are assigned in this non-alphabetical order because they turn out to correspond to the coefficients in the vector $\bv {a+b\phi}{c+d\phi}$; see Lemma \ref{lem:correspondence}.

\begin{lemma}
  \label{lem:rs-on-abcd}
  The effects of each $r_i$ on the lengths of cutting sequences, depending on the values of $a_\tau, b_\tau, c_\tau, d_\tau$, are as follows:
  \begin{align*}
    r_0 & : [a_\tau,b_\tau,c_\tau,d_\tau]\to [a_\tau + d_\tau, b_\tau + c_\tau + d_\tau, c_\tau, d_\tau]                                     \\
    r_1 & : [a_\tau,b_\tau,c_\tau,d_\tau]\to [b_\tau + d_\tau, a_\tau + b_\tau + c_\tau + d_\tau, a_\tau + d_\tau, b_\tau + c_\tau + d_\tau] \\
    r_2 & : [a_\tau,b_\tau,c_\tau,d_\tau]\to [b_\tau + c_\tau, a_\tau + b_\tau + d_\tau, b_\tau + d_\tau, a_\tau + b_\tau + c_\tau + d_\tau] \\
    r_3 & : [a_\tau,b_\tau,c_\tau,d_\tau]\to [a_\tau, b_\tau, b_\tau + c_\tau, a_\tau + b_\tau + d_\tau].
  \end{align*}
\end{lemma}

\begin{proof}
  Given a short periodic trajectory $\tau$ and its corresponding cylinder vector $v$, and cutting sequence $s$, given $i \in \{0, 1, 2, 3\}$, let $v' = \sigma_i(v)$ and $\tau'$ the short periodic trajectory in direction $v'$, and $s'$ the corresponding cutting sequence.

  Staring at the definition of the $r_i$'s in Definition~\ref{def:subs},
  we just need to count, for each letter $\ar 12$, $\ar 21$, $\ar 23$, $\ar 32$, $\ar 34$, $\ar 43$, $\ar 45$, $\ar 54$, how many of its appearances on the right hand side come from each type on the left-hand side.

  For example, for $r_0$, letters $\ar 12$ and $\ar 21$ come from letters $\ar 12$, $\ar 21$, $\ar 23$, $\ar 32$, consequently, \mbox{$a_{\tau'} = a_\tau + d_\tau$}.
\end{proof}

\subsection{Periods of periodic double pentagon trajectories}\label{sec:periods}

In this section we answer a simple question: Given a periodic direction on the double pentagon surface or on the pentagonal billiard table, what are the corresponding periods? The analogous result is well known for the square; a trajectory with slope $p/q$ on the square torus has period $p+q$,  and a trajectory with slope $p/q$ on the square billiard table has period $2(p+q)$ (Theorem \ref{thm:square}).

%
%

It turns out that the answer for the double pentagon is the same as for the square, if we set up our definitions correctly. First we must establish periodic directions:

\begin{lemma}
  Periodic directions in the double pentagon and in the pentagonal billiard table are exactly the images of the directions in $\qrf$ under the matrix $P$.
\end{lemma}

\begin{proof}
  A periodic direction in the double pentagon corresponds to a saddle connection vector $\mathbf{v}_P$ on the double pentagon. The image $P^{-1} \mathbf{v}_P$ of this vector on the golden L is a saddle connection on the golden L; see Figure \ref{fig:cutandpaste}. By Theorem \ref{thm:abcd}, saddle connection vectors on the golden L have components in $\Z[\phi]$, so their slopes are in $\qrf$. Going in the other direction, any direction in $\qrf$ is a saddle connection direction on the golden L, for some saddle connection vector $\mathbf{v}_L$, and the image $P \mathbf{v}_L$ of this vector on the double pentagon is a saddle connection on the double pentagon, and hence a periodic direction.
\end{proof}


\begin{theorem}\label{thm:saddle-connection}
  Given a vector $\vv\in\Sigma$ parallel to a periodic trajectory on the golden L, there are two cylinders of periodic trajectories on the L, one short and one long, in this direction. The short one is in $\Lambda$, and the long one is in $\phi\Lambda$. The one in $\Lambda$ is $\ell^{-1}\vv$, where $\ell = \ell_\vv$ from Proposition \ref{prop:terminates}. It lives in $\Z[\phi]^2$, i.e. it is of the form $\sv{a+b\phi}{c+d\phi}$ with $a,b,c,d$ nonnegative integers and $(a,b)\neq (0,0)$.
\end{theorem}

\begin{proof}
  This is a rephrasing of Proposition \ref{prop:scvectors}.
\end{proof}

We now give an expression for the period of a trajectory on the double pentagon, corresponding to a given periodic direction:

\begin{theorem}[Combinatorial Period Theorem]\label{thm:double}
  Given a vector $\vv = \sv{a+b\phi}{c+d\phi}\in\Lambda$, the corresponding short trajectory has holonomy vector $\vv$ and the long trajectory has holonomy vector $\phi\vv$. The combinatorial periods of the corresponding trajectories in the double pentagon are $2(a+b+c+d)$ and $2(a+2b+c+2d)$, respectively.
\end{theorem}

\begin{example}\label{example:double}
  To determine the (periodic) word corresponding to a given direction $\vv$, we use the (finite) sequence of sectors $a_\vv = \{k_0, k_1, \ldots,k_n\}$ to work backwards. The word corresponding to the direction $[1,0]$ is $w_n = \ar 12 \ar 21$ (see Figure \ref{four-dp}). We run this word through the transition diagram corresponding to $k_n$, to obtain the previous word $w_{n-1}$, and so on. For example, if $a_\vv = \{1,2,0\}$, to get the short trajectories, we run $w_3 = \ar 12 \ar 21$ through $r_1$,  run the result through $r_2$, and run the result of that through $r_0$, obtaining the sequences
  \begin{align*}
    w_3 & = \ar 12 \ar 21                                                                                                                                             \\
    w_2 & =\ar 34 \ar 45 \ar 54 \ar 43                                                                                                                                \\
    w_1 & =\ar 54 \ar 43 \ar 32 \ar 21 \ar 12 \ar 23 \ar 32 \ar 21 \ar 12 \ar 23 \ar 34 \ar 45                                                                        \\
    w_0 & = \ar 54 \ar 43 \ar 34 \ar 43 \ar 32 \ar 21 \ar 12 \ar 21 \ar 12 \ar 23 \ar 34 \ar 43 \ar 32 \ar 21 \ar 12 \ar 21 \ar 12 \ar 23 \ar 34 \ar 43 \ar 34 \ar 45
  \end{align*}
  in turn.

  We also use $a_\vv$ to find the direction. We compute $\sigma_0 \sigma_2 \sigma_1 \sv 10$, remembering that $\phi^2 = \phi+1$.
  \begin{align*}
    \sigma_0 \sigma_2 \sigma_1 \sv 10
     & = \bm 1 \phi 0 1 \bm \phi 1 \phi \phi \bm \phi \phi 1 \phi \bv 10 \\
     & = \bm 1 \phi 0 1 \bm \phi 1 \phi \phi  \bv \phi 1                 \\
     & =  \bm 1 \phi 0 1   \bv {\phi+2} {2\phi+1}                        \\
     & = \bv {4\phi+4}{2\phi+1}.
  \end{align*}
  We show the intermediary steps above because they give us the direction corresponding to each word: $w_3=12$ corresponds to direction $\sv 10$, \\ $w_2=3454$ corresponds to direction $\sv \phi 1$, \\ $w_1=543212321234$ corresponds to direction $\sv {\phi+2} {2\phi+1}$,  and \\ $w_0 = 5434321212343212123434$ corresponds to direction $\sv {4\phi+4}{2\phi+1}$.

  We have given undecorated sequences for ease of reading and counting. The reader can check that at each step, the length of the period is twice the sum of $a,b,c$, and $d$ in the vector $\sv {a\phi+b}{c\phi+d}$.
\end{example}

Now we develop some preliminary results that will help us to prove Theorem \ref{thm:double}.

\begin{lemma}\label{lem:buddies}
  Every decorated periodic cutting sequence consists of equal numbers of $a^b$ and $b^a$ for each $a,b\in\{1,2,3,4,5\}$.
\end{lemma}

\begin{proof}[Proof of Lemma \ref{lem:buddies} by contradiction]
  If the cutting sequence contained $n$ $\ar ab$s and fewer than $n$ $\ar ba$s, this would mean that the associated path on the transition diagram goes at least once from $a$ to $b$, without going back from $b$ to $a$, contradicting the fact that the path returns to where it started.
\end{proof}

\begin{proof}[Proof of Lemma \ref{lem:buddies} by induction]
  Every (short) periodic sequence can be obtained via a series of substitutions $r_i$ (see Equation (\ref{eqn:subs})) applied to the periodic sequence $(\ar 12 \ar 21)$. For each $i$, $r_i(\ar ab)$ is the same as $r_i(\ar ba)$ but in the reverse order and with the base and exponent reversed, so that each $\ar cd$ in the $i^{\text{th}}$ position in $r_i(\ar ab)$ is matched with $\ar dc$ in the $i^{\text{th}}$ position from the end in $r_i(\ar ba)$. Since the original sequence $(\ar 12 \ar 21)$ consists of pairs of this form, and each substitution preserves this property, every cutting sequence consists of pairs of this form.
\end{proof}

The inductive proof gives an additional result:

\begin{corollary}\label{cor:palindrome}
  Every periodic cutting sequence has several palindromic properties: It reads the same in the reverse order if
  \begin{enumerate}
    \item  the base and exponent are reversed, or if
    \item bases and exponents are treated the same, or if
    \item we ignore the exponents and remove the first digit (using the order given by applying the substitutions (\ref{eqn:subs}) to $1^2 2^1$ or $3^4 4^3$).
  \end{enumerate}
\end{corollary}

For example, the periodic cutting sequence $543212321234$ from Example \ref{example:double} can be rewritten in decorated form as $\ar 54 \ar 43 \ar 32 \ar 21 \ar12 \ar 23 \ar 32 \ar 21 \ar 12 \ar 23 \ar 34 \ar 45$, showing these properties.

\begin{lemma}
  \label{lem:sigmas-on-abcd}
  The matrices $\sigma_i$ have the following effects on a vector in $\R^2$ with entries in $\mathbf{Z}[\phi]^2$:
  \begin{align*}
    \sigma_0 \bv {a+b\phi}{c+d\phi} & = \bv {(a+d)+(b+c+d)\phi}{c+d\phi},             \\
    \sigma_1 \bv {a+b\phi}{c+d\phi} & = \bv {(b+d)+(a+b+c+d)\phi}{(a+d)+(b+c+d)\phi}, \\
    \sigma_2 \bv {a+b\phi}{c+d\phi} & =\bv {(b+c)+(a+b+d)\phi}{(b+d)+(a+b+c+d)\phi},  \\
    \sigma_3 \bv {a+b\phi}{c+d\phi} & = \bv {a+b\phi}{(b+c)+(a+b+d)\phi}.
  \end{align*}
\end{lemma}

\begin{proof}
  This is easily checked by multiplying out the matrices on the left hand side (each matrix $\sigma_i$ is given in Definition \ref{def:sigmas}), and applying the identity $\phi^2 = 1+\phi$.
\end{proof}

\begin{lemma}\label{lem:correspondence}
  %
  %
  Let $\tau$ be a short periodic trajectory on the double pentagon and $\mathbf{v} = \bv {a + b \phi}{c + d \phi}$ be the corresponding short saddle connection vector.  Let $a_\tau, b_\tau, c_\tau, d_\tau$ be the number of pairs of types of letters in the cutting sequence as described in Definition \ref{def:count}.
  %
  %
  Then $(a_\tau, b_\tau, c_\tau, d_\tau) = (2a, 2b, 2c, 2d)$.
\end{lemma}

\begin{proof}
  Corresponding to the vector $\vv$ is a unique sequence $a_\vv = (k_0, k_1, ..., k_n)$ of sectors as in Definition~\ref{def:algorithm} such that $\vv = \sigma_{k_0} ... \sigma_{k_n} \sv{1}{0}$.
  Correspondingly, $\tau = r_{k_0} ... r_{k_n} (\ar 12 \ar 21)$.
  The proof is by induction on $n$, and the induction step consists in observing that the
  effect of the $\sigma_i$'s on $(a, b, c, d)$ and of the $r_i$'s on $(a_\tau, b_\tau, c_\tau, d_\tau)$ match, by comparing Lemmas \ref{lem:rs-on-abcd} and \ref{lem:sigmas-on-abcd}.
\end{proof}

Now we are ready to prove the main theorem.

\begin{proof}[Proof of the Combinatorial Period Theorem \ref{thm:double}]
  The proof is by induction on depth in the tree.

  For the short trajectory, the base case is the vector $\vv=\sv 10 = \sv{1+0\phi}{0+0\phi}$, corresponding to cutting sequence $\ar 12 \ar 21$. The period is twice the sum of the coefficients, so the result holds.

  Lemma \ref{lem:correspondence} shows that each application of the $r_i$'s and the $\sigma_i$'s has an equivalent effect on the cutting sequence letters and on the vector coefficients, so the relationship is preserved.

  For the long trajectory, we start with cutting sequence $3^4 4^3$, and the combinatorics matches just as with the short trajectory.
\end{proof}

At the end of this section in \S\ref{subsec:extended}, we give an extended proof of the Combinatorial Period Theorem \ref{thm:double} that introduces some new ideas.

\subsection{The L-necklace, the Five Ls, and their Veech group}\label{sec:veech_group}

As discussed in \S\ref{subsec:necklace}, the \emph{Necklace}, which we will now call the \emph{P-necklace}, is a translation surface made of $10$ pentagons glued edge-to-edge, which is the unfolding of the pentagonal billiard table. It is a 5-fold cover of the double pentagon surface. Here, we introduce a different presentation of the same topological surface: the \emph{L-necklace} is a transformed version of the necklace under $P^{-1}$ (Definition \ref{def:stretch}), made of $5$ copies of the golden L (see Figure \ref{necklaces}).

\begin{figure}[!ht]
  \centering
  \includegraphics[height=200pt]{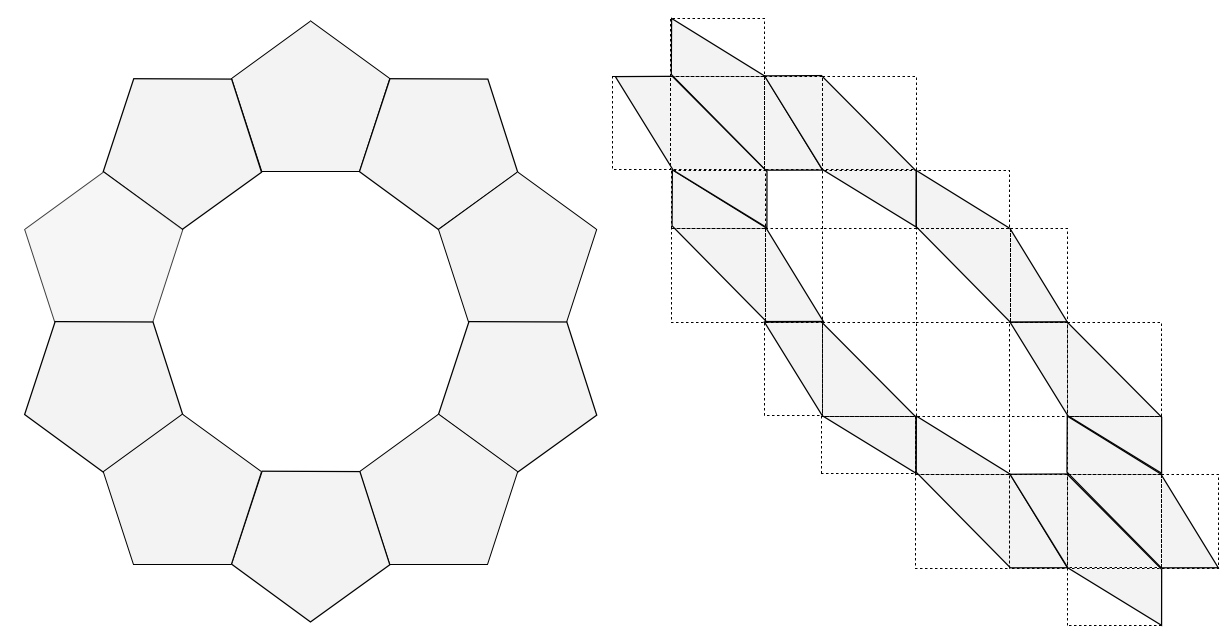}
\caption{The P-necklace and the L-necklace, which are affine images of each other. Each edge is glued to the parallel facing edge; for labeling, see Figure \ref{fig:necklace}. The dotted rectangles show the underlying structure of golden Ls. \label{necklaces}}
\end{figure}

We study the necklace because a cutting sequence on the necklace can be projected down to a cutting sequence on the double pentagon via a 5-to-1 edge label map. Also, the necklace can be folded into the pentagon billiard table via a 10-to-1 edge label map.

It is useful to reassemble the $L$-necklace into the underlying golden $L$s, as in Figure \ref{five-ls}. We call this presentation the \emph{Five $L$s}.

\begin{figure}[!ht]
  \centering
  \begin{tikzpicture}[scale=1.75]
  \foreach \j in {0,1,2,3,4} {
    \draw [xshift=1.75*\j cm, semithick]
      (0, 0)
      -- node [above=-2pt] {$D'_{\j}$} (1, 0)
      -- node [above=-2pt] {$B'_{\j}$} (1.618, 0)
      -- node [left=-2pt] {$A_{\j}$} (1.618, 1)
      -- node [below=-2pt] {$B_{\j}$} (1, 1)
      -- node [left=-2pt] {$C_{\j}$} (1, 1.618)
      -- node [below=-2pt] {$D_{\j}$} (0, 1.618)
      -- node [right=-2pt] {$C'_{\j}$} (0, 1)
      -- node [right=-2pt] {$A'_{\j}$} (0, 0)
    ;
    \draw [xshift=1.75*\j cm, gray, very thin]
      (1, 0) -- (1, 1) -- (0, 1);
  };
  \end{tikzpicture}
  \caption{Five Ls, from which we build surfaces $f(a,b,c,d)$ by gluings defined in Definition \ref{fourtuple}. \label{five-ls}}
\end{figure}

\begin{definition}\label{fourtuple}
  We define a family of translation coverings of degree 5 of the golden L. For that, consider five copies $L_0, \dots L_4$ of the golden L, with indices thought of modulo 5. Label the sides of each $L_j$ as $A_j, B_j, C_j, D_j, A'_j, B'_j, C'_j, D'_j$ as in Figure \ref{five-ls}. For any $(a,b,c,d)\in(\Z/5\Z)^4$, define $f(a,b,c,d)$ as the surface obtained from $L_0\dots L_4$ by gluing sides $A_j$ to $A'_{j+a}$,
  $B_j$ to $B'_{j+b}$,
  $C_j$ to $C'_{j+c}$,
  $D_j$ to $D'_{j+d}$, for all $j$ in $\Z/5\Z$.

\end{definition}

In this notation, the $L$-necklace is $f(1,3,2,4).$ This is easy to check.

\begin{definition}
  Let $T$ be the horizontal shear $\sm 1{\phi}01=\sigma_0$ and let $R$ be the counter-clockwise quarter-turn $\sm 0{-1}10$.
\end{definition}

$T$ and $R$ generate the Veech group $\Gamma$ of the golden L, which is the Hecke group $G_5$. We analyze $T$ and $R$ with respect to their effects on the surfaces $f(a,b,c,d)$. $T\cdot f(a,b,c,d)$ records the gluings that result from applying the shear to the surface with gluings $(a,b,c,d)$, and similarly for $R\cdot f(a,b,c,d)$.

\begin{lemma}\label{lem:relabeling} By applying the matrices $T$ and $R$ and relabeling, we obtain the relations $T\cdot f(a,b,c,d) = f(a, b-a, c, d-a-c)$ and $R\cdot f(a,b,c,d) = f(-d,c,-b,a)$.

  If $a=b=c=d=0$, we get a disconnected surface with five connected components, each of which is a copy of the golden L surface. 
Furthermore, we can eliminate redundancies modulo $5$ with the following reductions, where the quotients are all taken modulo $5$:

  \begin{itemize}
    \item If $a\neq 0$, then dividing by $a$, we get $f(a,b,c,d) =  f(1,b/a,c/a,d/a)$.
    \item If $a=0$ and $b\neq 0$, then dividing by $b$, we get $f(0,b,c,d) = f(0,1,c/b,d/b)$.
  \end{itemize}
  These relabelings give a canonical representation of each surface.
\end{lemma}

\begin{proof}
  The first two relations are a simple exercise in translation surfaces. The reductions are obtained by relabeling the five Ls in each $f(a,b,c,d)$.
\end{proof}

\begin{proposition} \label{prop:TRgraph}
  The $\Gamma$-orbit of the $L$-necklace consists of six surfaces (Figure \ref{sixnodes}).
\end{proposition}

\begin{proof}
  We compute the orbit of the $L$-necklace under $\Gamma$ using $T$ and $R$. The gluings and relations in Figure \ref{sixnodes} are obtained by simple calculations in translation surfaces, and relabelings. No two of these six surfaces can be the same, by the uniqueness after relabeling from Lemma \ref{lem:relabeling}.
\end{proof}

\begin{figure}[!ht]

  \centering
  \begin{tikzpicture}[description/.style={fill=white,inner sep=2pt}]
    \matrix (m) [matrix of math nodes, row sep=3em,
      column sep=2.5em, text height=1.5ex, text depth=0.25ex]
    { f(1,3,2,4) & f(1,4,2,2) &            &            \\
                 &            & f(1,0,2,0) & f(0,1,0,3) \\
      f(1,2,2,1) & f(1,1,2,3) &            &            \\};

    \path[->] (m-1-1) edge [bend right=15] node [left]{$T$}(m-3-1);
    \path[<->] (m-1-1) edge [bend left=15] node [right]{$R$}  (m-3-1);

    \path[->] (m-3-1) edge node [above] {$T$} (m-3-2);
    \path[->] (m-3-2) edge node [above] {$T$}  (m-2-3);
    \path[->] (m-2-3) edge node [below] {$T$} (m-1-2);
    \path[->] (m-1-2) edge node [below] {$T$}  (m-1-1);

    \path[<->] (m-2-3) edge node [above] {$R$}  (m-2-4);
    \path[<->] (m-2-4) [bend left=15] (m-2-4);

    \draw [->,line width=.5pt] (5.05,-0.1) arc[x radius=0.3cm, y radius =.3cm, start angle=-160, end angle=160] ;
    \node at (5.9,0) {$T$};

    \draw [->,line width=.5pt] (-1.1,1.8) arc[x radius=0.3cm, y radius =.3cm, start angle=-60, end angle=220] ;
    \node at (-1.2,2.6) {$R$};
    \draw [->,line width=.5pt] (-1.5,-1.9) arc[x radius=0.3cm, y radius =.3cm, start angle=-220, end angle=60] ;
    \node at (-1.2,-2.7) {$R$};


  \end{tikzpicture}

  \caption{The effects of the shear $T$ and the rotation $R$ on the gluings of the Five Ls.
      \label{sixnodes}}
\end{figure}

We name the five nodes that are cyclically related by $T$ to be $A, B, C, D, E$, counter-clockwise starting with the top left, and the node $f(0,1,0,3)$ to be $F$, as shown in Figure \ref{colorgraph}.




\subsection{Symmetries of trajectories via the group structure}\label{sec:symm_group}

Among periodic trajectories on the pentagonal billiard table, some are rotationally symmetric  (Figure \ref{fig:symmetric}), some are not (Figure \ref{fig:asymmetric}). In this section we relate this property to some paths in a graph encoding the action of $\sigma_0\dots \sigma_3$ on surfaces $A, B, C, D, E, F$.


To each periodic direction, we associate a starting horizontal trajectory, with cutting sequence either $12$ (short) or $34$ (long) on the double pentagon, and a tree word. For example, the tree word $120$ means that we apply the product $\sigma_0 \sigma_2 \sigma_1$ to $\sv 10$ in order to get the associated direction. In order to extract information about where we end up in the graph depending on its tree word, we express the $\sigma_i$ in terms of $T$ and $R$ and re-draw the graph in Figure \ref{colorgraph}.

\begin{lemma}
  Expressing each $\sigma_i$ in terms of $T$ and $R$ and following them around the directed graph in Figure \ref{sixnodes} yields the directed graph in Figure \ref{colorgraph}, which tells us how the $\sigma_i$s permute the surfaces $A$, $B$, $C$, $D$, $E$ and $F$.
\end{lemma}

\begin{proof}
  We have the following relationships:
  \begin{align*}
    \sigma_0 & = T       \\
    \sigma_1 & = TRT     \\
    \sigma_2 & = TRTRT   \\
    \sigma_3 & = TRTRTRT \\
  \end{align*}
  Starting at each node in the graph in Figure \ref{sixnodes}, we determine the effect of each $\sigma_i$ by following the arrows corresponding to $T$ and $R$. This yields the directed graph in Figure \ref{colorgraph}.
\end{proof}

\begin{figure}[!ht]

  \begin{center}
    \begin{tikzpicture}[scale=1.5]
      \node (A) at (0,2) {A};
      \node (B) at (0,0) {B};
      \node (C) at (2,0) {C};
      \node (D) at (3,1) {D};
      \node (E) at (2,2) {E};
      \node (F) at (5,1) {F};
      \draw [dotted, ->, black, very thick] (A) to [bend left=10] (B);
      \draw [dashed, ->, blue]  (A) to [bend left=30] (B);
      \draw [->, red]   (A) to [bend right=10] (B);
      \draw [->, green, very thick] (A) to [bend right=30] (B);
      \draw [dotted, ->, black, very thick] (B) to (C);
      \draw [dashed, ->, blue]  (B) to (E);
      \draw [->, red]   (B) to [bend right=40] (F);
      \draw [->, green, very thick] (B) to (D);
      \draw [dotted, ->, black, very thick] (C) to [bend right=10] (D);
      \draw [dashed, ->, blue]  (C) to (A);
      \draw [->, red]   (C) to [bend right=10] (E);
      \draw [->, green, very thick] (C) to [bend right=10] (F);
      \draw [dotted, ->, black, very thick] (D) to [bend right=10] (E);
      \draw [dashed, ->, blue]  (D)  .. controls +(.7,-.3) and +(.7,.3) ..  (D);
      \draw [->, red]   (D) to [bend right=10] (C);
      \draw [->, green, very thick] (D) to (A);
      \draw [dotted, ->, black, very thick] (E) to (A);
      \draw [dashed, ->, blue]  (E) to [bend right=10] (F);
      \draw [->, red]   (E) to [bend right=10] (D);
      \draw [->, green, very thick] (E) to [bend right=10] (C);
      \draw [dotted, ->, black, very thick] (F) .. controls +(1,-1) and +(1,1) .. (F);
      \draw [->, green, very thick]  (F) to [bend right=10](E);
      \draw [->, red]   (F) to  [bend right=40] (A);
      \draw [dashed, ->, blue] (F) to   [bend right=10] (C);
    \end{tikzpicture}
  \end{center}

  \caption{The location in the Veech group of the Five Ls, depending on the $\sigma_i$. Start at node $A$ and follow arrows \emph{backwards} according to the tree word for a given direction, read \emph{backwards}, to determine its associated symmetry. The colors are as follows: dotted black is $\sigma_0$, thick green is $\sigma_1$, thin red is $\sigma_2$, and dashed blue is $\sigma_3$. Landing on $F$ indicates only reflection symmetry; landing on the other five indicates rotation symmetry as well. \label{colorgraph}}
\end{figure}

\begin{theorem}\label{thm:symm}
  Given a periodic trajectory and the tree word associated to its direction, start at $A$ in the diagram in Figure \ref{colorgraph}, and follow the arrows \emph{backwards} according to the tree word, read \emph{backwards}. If it ends on $A$ through $E$, the trajectory has order $5$ rotational symmetry; if it lands on $F$, the trajectory does not have rotational symmetry.
\end{theorem}

\begin{proof}
  Every periodic direction has an associated long saddle connection vector. A finite product $\sigma_i$s, for example $\sigma_0\sigma_2\sigma_1$ for tree word 120, which is an element of the Veech group of the golden L, takes $\sv 10$ to this saddle connection vector. By Proposition \ref{prop:TRgraph}, that element takes the surface to one of $A$, $B$, $C$, $D$, $E$ or $F$.

  Thus the inverse, e.g. $(\sigma_0\sigma_2\sigma_1)^{-1} = {\sigma_1}^{-1}{\sigma_2}^{-1}{\sigma_0}^{-1}$, of the element takes the saddle connection vector to a horizontal vector. Applying this inverse corresponds to reading the tree word backwards (021 for our example) and following the arrows backwards, since we are using ${\sigma_i}^{-1}$. We look at the two cylinders of trajectories on the Five Ls corresponding to this horizontal trajectory.


  For $A$, $B$, $C$, $D$ and $E$, $a=1$ and $c=2$, so the Five Ls are arranged horizontally as in Figure \ref{five-ls}. Thus the horizontal cylinders are five times as wide as in one golden L, and the horizontal trajectory traverses all of them before repeating. This means that the pentagon trajectory traverses all five copies of the double pentagon in the necklace, corresponding to all five orientations (rotations by $2\pi/5$), one after the other, so the trajectory has order 5 rotational symmetry, and its length is five times as long as in the double pentagon.

  For $F$, $a=0$ and $c=0$, so the five Ls are arranged vertically, so a horizontal cylinder is the same as in a single golden L. Thus the golden $L$ trajectory lifts to five distinct translated cylinders of trajectories, one for each orientation, which are images of each other under rotation by $2\pi/5$. So trajectories in case $F$ fail to individually have rotational symmetry.
\end{proof}

\begin{figure}[!ht]
  \centering
  \includegraphics[width=0.24\textwidth]{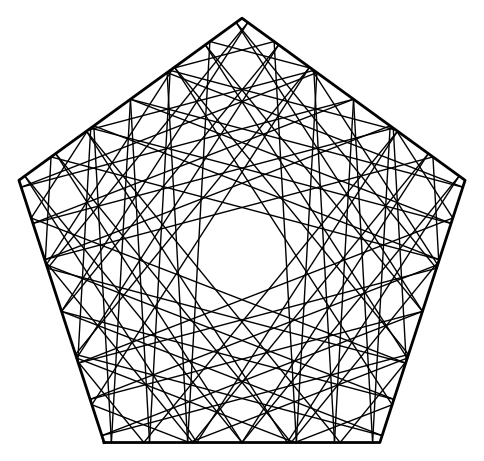}\
  \includegraphics[width=0.24\textwidth]{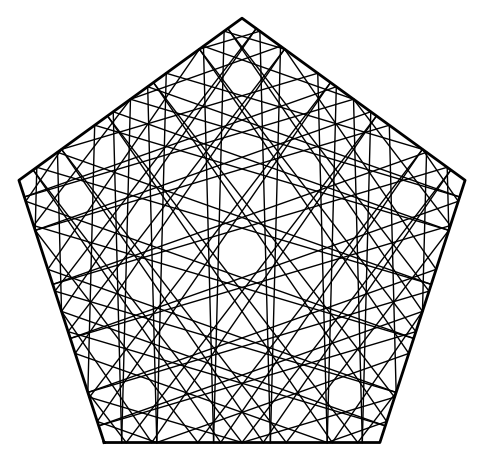} \
  \includegraphics[width=0.24\textwidth]{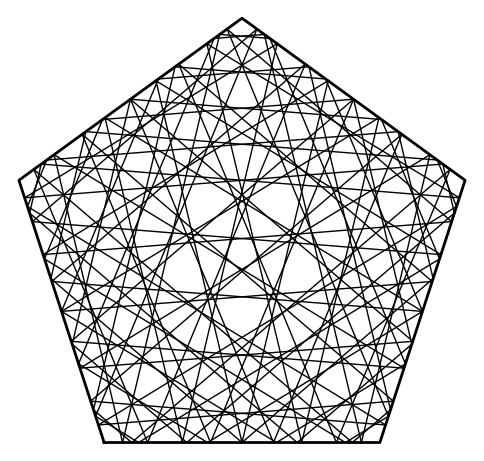}\
  \includegraphics[width=0.24\textwidth]{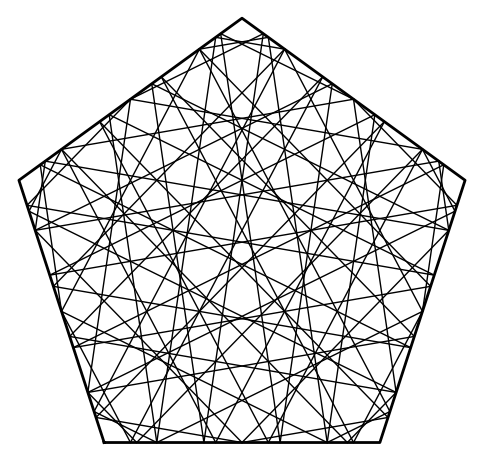}
\caption{Examples of periodic paths on the pentagon with the full symmetry of $D_5$.  These trajectories are 101-short, 12-long, 1000-long and 23-long, respectively, corresponding to the core curve of the cylinder. Using a non-core curve breaks the bilateral symmetry and results in a picture with only (usually very subtle) rotational symmetry.\label{fig:symmetric}}
\end{figure}

\begin{corollary}\label{cor:sixths}
  When enumerating periodic orbits in the pentagon billiard table by depth in the tree, the asymptotic proportion that have both rotational symmetry and reflection symmetry is $5/6$ (see Figure \ref{fig:symmetric}), while the asymptotic proportion with only reflection symmetry is $1/6$ (see Figure \ref{fig:asymmetric}).
\end{corollary}

\begin{proof}
  The stochastic matrix for the graph in Figure \ref{colorgraph} is:
  \[
    \left[\begin{matrix}
        0 & 0 & 1/4 & 1/4 & 1/4 & 1/4 \\1&0&0&0&0&0\\0&1/4&0&1/4&1/4&1/4\\0&1/4&1/4&1/4&1/4&0\\0&1/4&1/4&1/4&0&1/4\\0&1/4&1/4&0&1/4&1/4
      \end{matrix}\right].
  \]
  The $(i,j)$ entry of this matrix gives among paths of length 1 starting from node $i$ the proportion that end at node $j$. Powers of this matrix give the analogous proportions for paths of other lengths. Powers of this matrix quickly approach the matrix all of whose entries are $1/6$. Thus a path of length $n$ asymptotically visits nodes equally, so $1/6$ of paths end on vertex $F$, and $5/6$ of them end on the other vertices.
\end{proof}

\begin{figure}[!ht]
  \centering
  \includegraphics[width=0.24\textwidth]{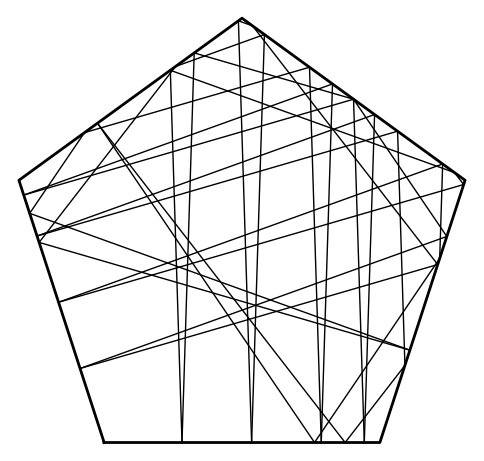} \
  \includegraphics[width=0.24\textwidth]{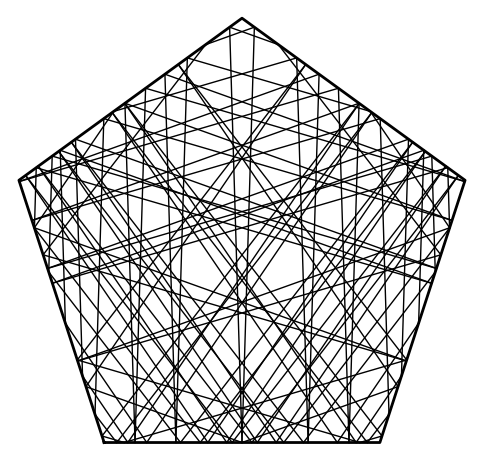}\
  \includegraphics[width=0.24\textwidth]{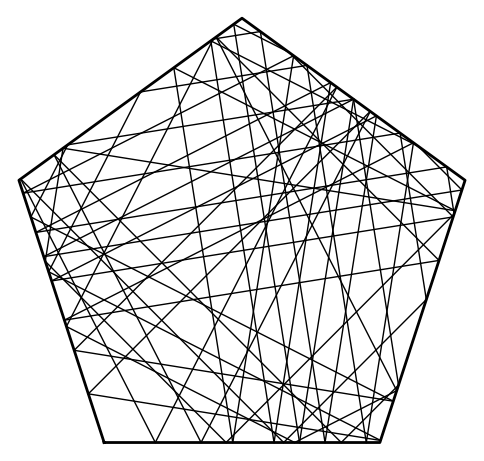}\
  \includegraphics[width=0.24\textwidth]{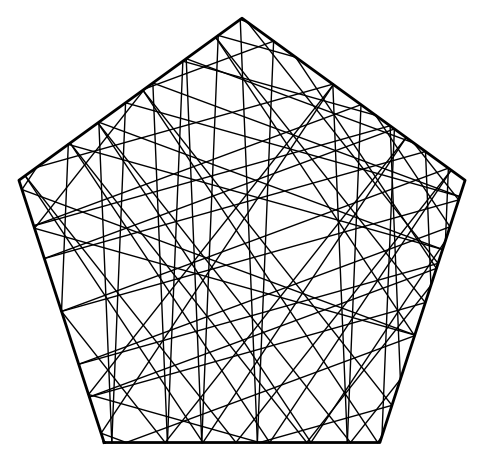}
\caption{Examples of periodic paths on the pentagon with only bilateral symmetry.  These are the short trajectories with tree words 121, 1112, 1223 and 2011, respectively, corresponding to the core curve of the cylinder. Using a non-core curve breaks the bilateral symmetry and results in a picture with no symmetry at all. \label{fig:asymmetric}}
\end{figure}

\begin{question}\label{q:five-sixths}
  Are the proportions also $5/6$ and $1/6$ when enumerating in other ways, such as by $\ell_1$, $\ell_2$ or $\ell_\infty$ norm of the associated saddle connection vector?
\end{question}

Our Theorem \ref{thm:symm} gives a method for following the tree word around a directed graph to determine the symmetry type of a trajectory. Compare this to the approach in \cite{DFT}, which uses a congruence condition on the saddle connection vector to determine the symmetry:

\begin{theorem}[\cite{DFT}, Theorem 11]\label{thm:dft}
  Given a periodic direction on the pentagon billiard table with long saddle connection vector $[a+b\phi, c+d\phi]$, let $L = 2(a+b+c+d)$. The period of the corresponding pentagon billiard trajectory is:
  $\begin{cases} L  & \mbox{if } (d-b)+2(c-a) \equiv 0     \\
      5L & \mbox{if } (d-b)+2(c-a) \not\equiv 0\end{cases} \pmod{5}.$
\end{theorem}


\subsection{Symmetries of trajectories via geometry}\label{sec:symmetries}

In this section, we develop geometric methods for understanding billiard paths using the necklace, to give yet another characterization of billiard path symmetry (Theorem \ref{thm:geom-symm}).

Given a trajectory on the double pentagon, each time the trajectory crosses an edge, we fold along that edge, and thereby obtain the corresponding trajectory on the pentagon billiard table. It turns out that, for periodic trajectories, there are two possible outcomes to this process.

\begin{theorem}\label{thm:geom-symm}
  For a periodic trajectory $p$ on the double pentagon surface corresponding to the core curve of a cylinder (i.e. bouncing off of the midpoint of an edge), two possibilities arise:
  \begin{enumerate}
    \item The preimage on the necklace is the union of five periodic trajectories, which each project one-to-one to $p$. In this case, the billiard trajectory has the same length as $p$, and only reflection symmetry (see Figure \ref{fig:asymmetric} and Figure \ref{fig:necklace-vert-horiz}(a)).
    \item The preimage on the necklace is a single periodic trajectory, whose length is five times that of $p$. In this case, the billiard trajectory also has length five times that of $p$, and it has both reflection and rotation symmetries (see Figure \ref{fig:symmetric} and Figure \ref{fig:necklace-vert-horiz}(b)).
  \end{enumerate}
\end{theorem}

\begin{figure}[!ht]
  \centering
  \includegraphics[width=0.46\textwidth]{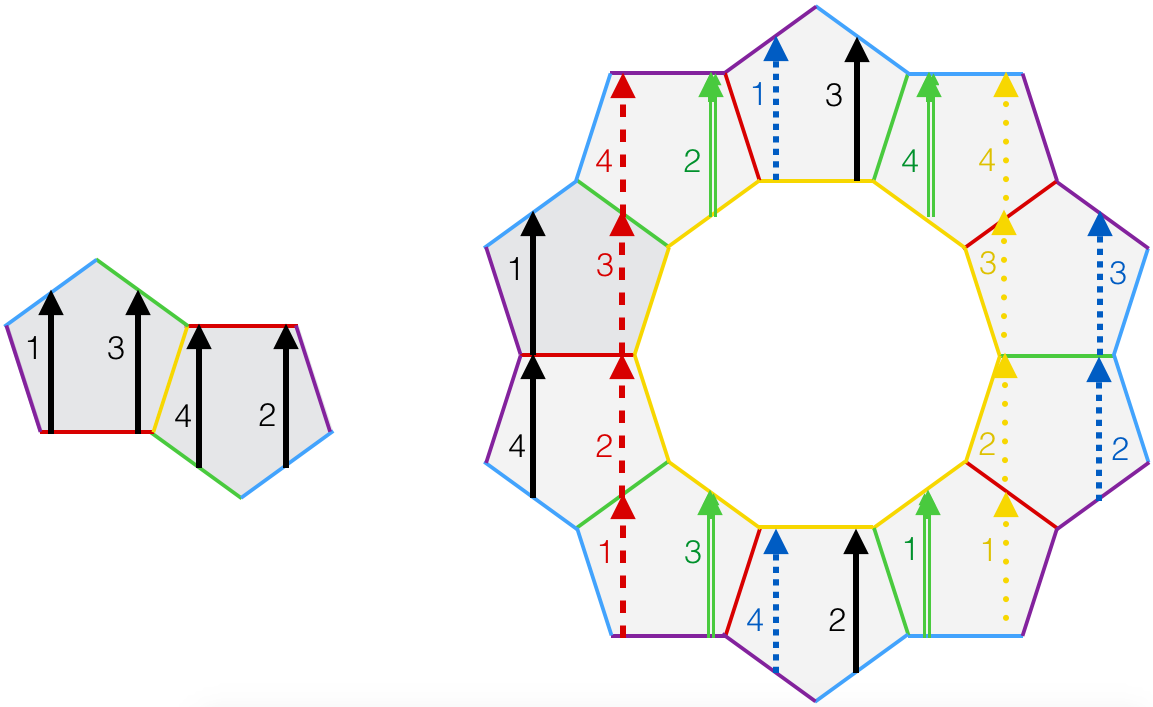} \hfill
  \includegraphics[width=0.46\textwidth]{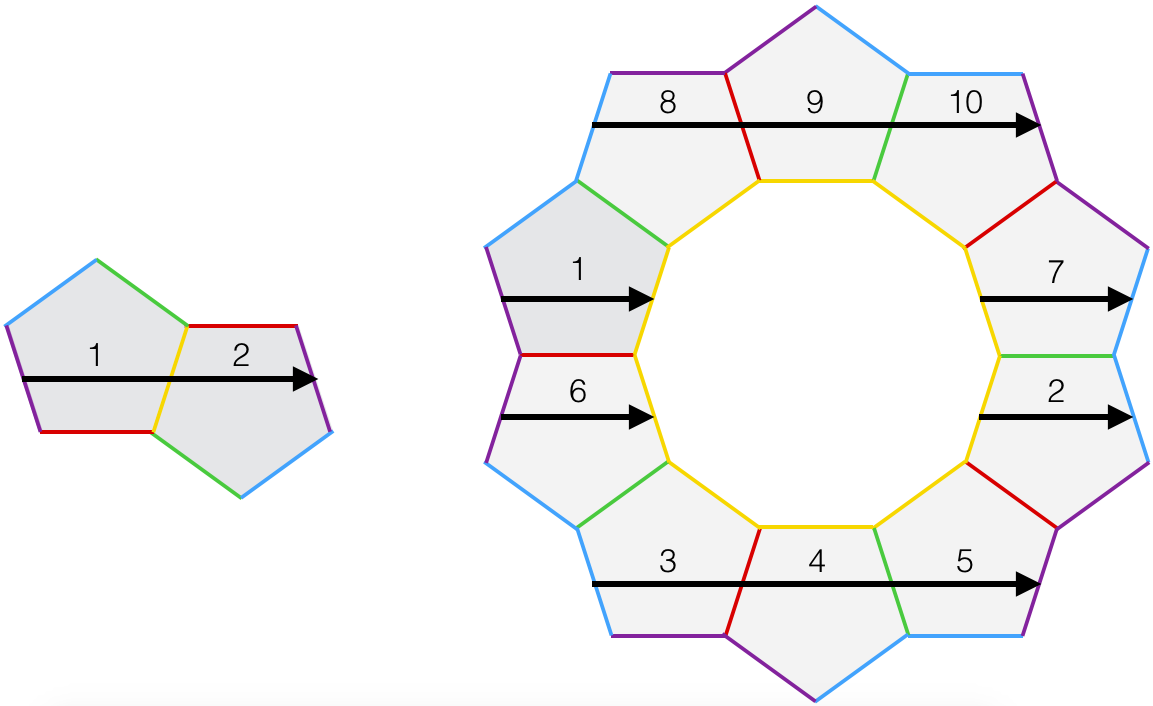}
\caption{Trajectories on the double pentagon lift to two distinct types of trajectories on the necklace. (a) The vertical trajectory lifts to five trajectories of the same length. (b) The horizontal trajectory lifts to a single trajectory five times as long. } \label{fig:necklace-vert-horiz}
\end{figure}

\begin{proof}
  Without loss of generality, we may assume that $p$ is the representative of its homotopy class that passes through the midpoint of two sides $s,t\in\{1,2,3,4,5\}$. Removing the midpoint of $s$, the trajectory is an open segment on the double pentagon.
  When we lift this segment to the necklace, it gives us five corresponding segments, which join midpoints of sides parallel to $s$. When we add in the midpoints, there are two possibilities, by regularity of the cover and the fact that 5 has two factors, 1 and 5. These possibilities are illustrated by the vertical and horizontal trajectories, respectively:
  \begin{enumerate}
    \item We get five disjoint loops on the necklace, each one going through midpoints of sides that are lifts of $s$, each one passing through the midpoint of an edge that is the unfolded image of a different edge of the billiard table. Thus, when the necklace is folded to yield the billiard table, these map to five trajectories that are rotations of each other, each one passing twice through the midpoint of exactly one edge. Thus such trajectories have reflection symmetry across the perpendicular bisector to that edge, but no rotation symmetry.
    \item We get a single loop of five times the length of $p$, which passes through the midpoints of all of the edges of the necklace that are lifts of $s$. Thus, when the necklace is folded to yield the billiard table, these fold to a trajectory that passes through the midpoint of each edge twice, and has reflection symmetry across the perpendicular bisector to every edge; thus, such trajectories have rotation symmetry as well.
  \end{enumerate}
\end{proof}

Examples of pentagon billiard trajectories that lift to a necklace trajectory five times as long as the double pentagon trajectory are in Figure \ref{fig:symmetric}. By construction, they have rotational symmetry, so along with the bilateral symmetry, they have the full symmetry of the dihedral group $D_5$.
Examples of pentagon billiard trajectories that lift to five copies of the double pentagon trajectory on the necklace are in Figure \ref{fig:asymmetric}.  Such paths have only bilateral symmetry.

\begin{lemma}\label{lem:primitive}
  Vectors in $\Lambda$ are primitive, in the sense that the only common factors of $a+b\phi$ and $c+d\phi$ are units in $\mathbf{Z}[\phi]$.
\end{lemma}

\begin{proof}
  Since $\sv{a+b\phi}{c+d\phi} = \sigma \sv 10$ for some $\sigma$ in the Veech group of the golden L, $\sv{a+b\phi}{c+d\phi}$ is the first column of a matrix with determinant 1. Thus we have the Biezout relation
  \[
    (a+b\phi)\cdot x +(c+d\phi)\cdot y = 1.
  \]
  Therefore any common divisor of $a+b\phi$ and $c+d\phi$ divides 1, so it is a unit.
\end{proof}

\begin{lemma}
  A primitive periodic orbit on the double pentagon has a primitive cutting word.
\end{lemma}

Here ``primitive periodic orbit'' means that we traverse exactly one period of the trajectory, and ``primitive cutting word'' means that the word does not consist of repeats of a shorter word.

\begin{proof}
  Suppose the cutting word is not primitive. Then the corresponding holonomy vector $[a+b\phi, c+d\phi]$ is not primitive as a vector $[a,b,c,d]\in\R^4$. Since by Lemma \ref{lem:primitive} vectors $[a+b\phi, c+d\phi]$ are always primitive, this is impossible.
\end{proof}

\begin{definition}
  There are two cylinders of periodic billiard trajectories parallel to the sides of the pentagon. We call the central trajectory in the short cylinder the \emph{mini pentagon} and the central trajectory in the long cylinder the \emph{star}. They are illustrated in Figure \ref{fig:two-bounces}.
\end{definition}

\begin{lemma}\label{lem:double-v}
  If a periodic billiard trajectory passes through the midpoint of an edge once, it passes through it twice, except for the mini pentagon and the star.
\end{lemma}

One can observe this in the trajectories in Figures \ref{fig:nice_examples}, \ref{fig:buddies}, \ref{fig:symmetric}, \ref{fig:asymmetric}.

\begin{proof}
  Suppose that a trajectory passes through the midpoint of an edge; without loss of generality, let it be the horizontal edge. We can start at this midpoint and follow the trajectory forward and backward from the bounce until they have each traversed half the length of the trajectory (Figure \ref{fig:two-bounces}). The two half-trajectories are reflection-symmetric across the vertical line of symmetry of the pentagon. Thus when they meet again, it is along this line of symmetry. If they meet at angle $\pi$, the segments that meet are horizontal, so the trajectory is parallel to the edges of the pentagon. Otherwise, the trajectory changes direction at the meeting point, so the meeting point must be a bounce, which must therefore be at the same midpoint.
\end{proof}

\begin{figure}[!ht]
\centering
\includegraphics[height=0.25\textwidth]{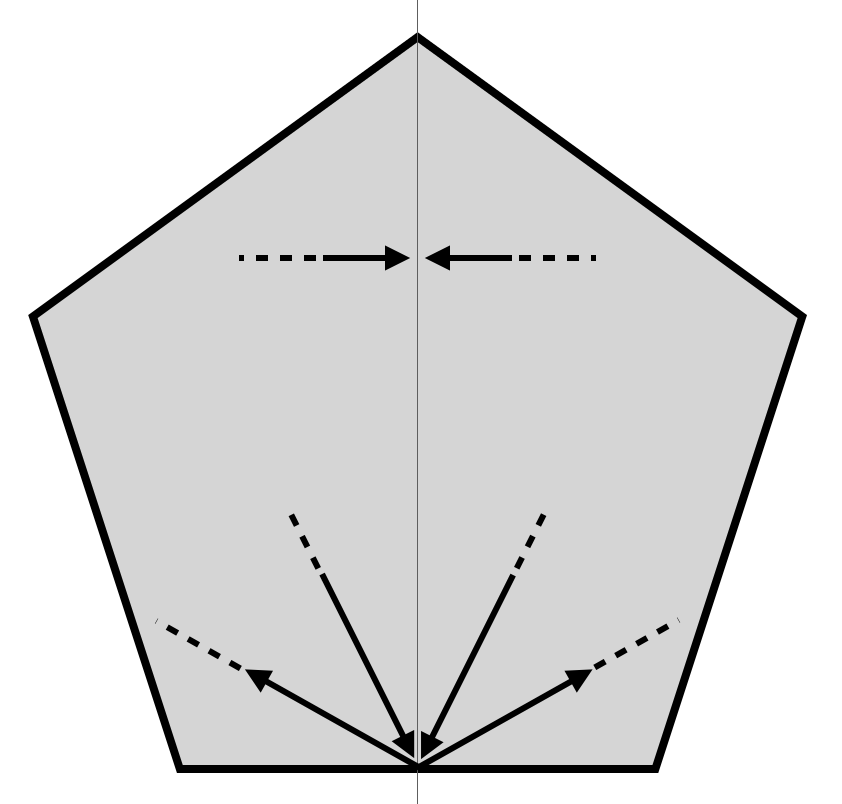}
\qquad\qquad
\includegraphics[height=0.25\textwidth]{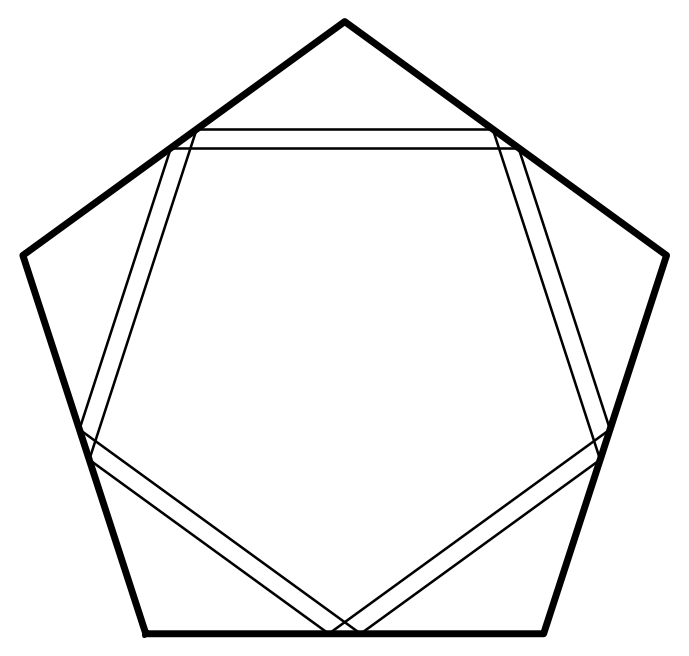}
\includegraphics[height=0.25\textwidth]{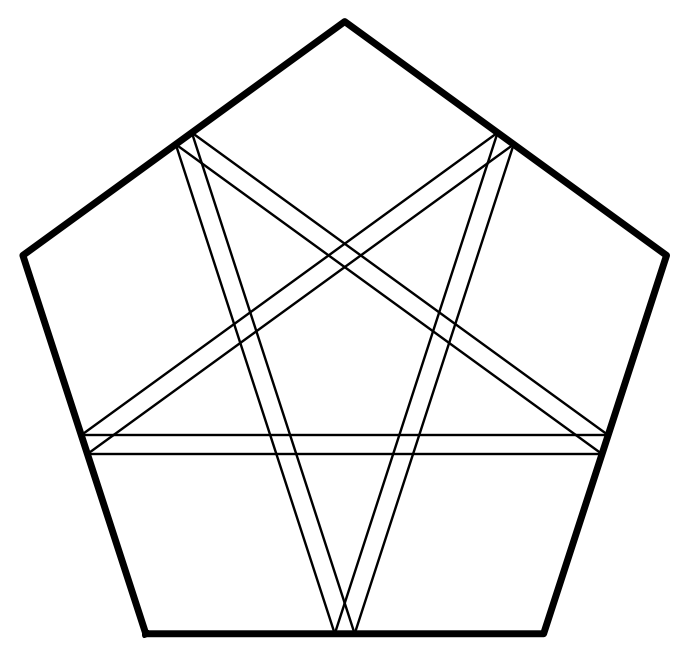}

\caption{Left: The two possibilities for a trajectory through a midpoint (Lemma \ref{lem:double-v}). Right: Representatives of the two cylinders of trajectories parallel to an edge of the pentagon; the core curves are the mini pentagon and star. This illustrates the length doubling proven in Lemma \ref{lem:corecurve}.} \label{fig:two-bounces}
\end{figure}

\begin{lemma}\label{lem:only-odd}
  The mini pentagon and the star, which have period 5, are the only trajectories with odd period.
\end{lemma}

\begin{proof}
  We consider the central periodic trajectory in a cylinder, which passes through a midpoint; without loss of generality, let it be of the horizontal side. If the trajectory has a horizontal segment, it must be the mini pentagon or the star. Otherwise, none of its segments are horizontal, so each one has a mirror-symmetric segment also part of the orbit. Therefore, the trajectory has an even number of segments.
\end{proof}

\begin{lemma}\label{lem:corecurve}
  For a periodic pentagon billiard trajectory, every trajectory in the homotopy class is the same length, except for the homotopy classes of the mini pentagon and the star, for which the trajectory through the midpoints is half as long as all the others.
\end{lemma}

\begin{proof}
  Consider a cylinder of periodic trajectories, and its central trajectory, with a designated direction of travel. We may assume that it passes through the midpoint of the horizontal side. Then consider a parallel trajectory inside the cylinder, hitting the horizontal side slightly to the left of its midpoint (see Figure \ref{fig:two-bounces}). At each bounce, it changes which side of the central trajectory it is on. Thus if the period of the central trajectory is even, the parallel trajectory closes up at the same time. If the period of the central trajectory is odd, the parallel trajectory has to repeat the loop before closing up, so its period is twice that of the central trajectory. By Lemma \ref{lem:only-odd}, the mini pentagon and star are the only periodic trajectories with odd period.
\end{proof}

In the introduction of \cite{DFT}, the authors mention that for a periodic billiard in the pentagon, ``if the period is odd, nearby parallel trajectories have the length and the period twice as large.'' Lemma \ref{lem:corecurve} shows that the \emph{only} such trajectories are the mini pentagon and the star.

\begin{proposition}\label{prop:corecurve}
  The length of a periodic billiard sequence corresponding to a trajectory with rotational symmetry is five times the length of the corresponding cutting word for the double pentagon, except for the horizontal trajectory, when it is half of this.
\end{proposition}

\begin{proof}
  This follows directly from Lemma \ref{lem:corecurve}.
\end{proof}

\subsection{Alternative proof of the Combinatorial Period Theorem}\label{subsec:extended}

We now give an alternative proof of the Combinatorial Period Theorem \ref{thm:double}, using $4$-dimensional integer vectors rather than $2$-dimensional vectors over $\Z[\phi]$. This perspective gives a faster way of computing the orbits because it works over the integers.

\begin{definition}
  Let $B =\{\ar 12, \ar 21, \ar 23, \ar 32, \ar 34, \ar 43, \ar 45, \ar 54\}$ and let $B^*$ be the set of all finite words in $B$. Note that our periodic sequences consist of a language within $B^*$ of words that are admissible, i.e. each symbol is compatible with the one before it (e.g. $\ldots\ar ab \ar bc\ldots$).

  We define two maps (recall Definition \ref{def:count} for $a_w$ etc.):
  \begin{align*}
    m: \N[\phi]^2         & \to \N^4          &  & n: B^*  \to \N^4                                  \\
    \bv{a+b\phi}{c+d\phi} & \mapsto  \fv abcd &  & \ \ \ \ \ \ \  \ \ w \mapsto \fv {a_w}{b_w}{c_w}{d_w}
  \end{align*}
  Since $\Z[\phi]$ has $\{1,\phi\}$ as a $\Z$-module basis, $m$ is invertible, so its inverse $m^{-1}: \N^4 \to \N[\phi]^2 $ is also defined.

  If $\vv=\bv{a+b\phi}{c+d\phi} \in\Lambda$, define its complexity as $c(\vv) = a+b+c+d$.

  Given a periodic word $w\in B^*$, define $\ell(w)$ to be its length.
\end{definition}

\begin{proof}[Extended proof of Theorem \ref{thm:double}]
  Now Theorem \ref{thm:double} is: If $\vv$ is the vector in $\Lambda$ corresponding to a periodic trajectory with associated cutting sequence $w$, then $\ell(w)=2c(\vv)$. In our new notation, we can state a stronger version: $n(w) = 2m(\vv)$, which we prove here.

  The proof is by induction on the depth in the tree, and the base case is the same as in the previous proof: the result holds for the vector $\vv = \sv 10$ and for the associated trajectory $w = \ar 12 \ar 21$. This is because in the new notation, $n(w) = \sfv 2000$ and $m(\vv) = \sfv 1000$, so $n(w) = 2m(\vv)$.

  Define $\hat{\Lambda} = m(\Lambda)$, which is the set of vectors $\sfv abcd$ such that $\bv{a+b\phi}{c+d\phi}\in\Lambda$.

  Now for each $i=0,1,2,3$, define  $\hat{\sigma_i}:\hat{\Lambda}\to\hat{\Lambda}$ such that $\hat{\sigma_i} = m^{-1}\circ \sigma_i \circ m$. These express the effects of the matrices $\sigma_i$, just in $4$-vector coefficient coordinates instead of $2$-vector coordinates.

  These are linear maps, and we can compute them explicitly from Lemma \ref{lem:sigmas-on-abcd}:

  $\hat{\sigma_0}= \begin{bmatrix} 1&0&0&1  \\ 0&1&1&1 \\ 0&0&1&1 \\ 0&0&0&1 \end{bmatrix}, \
    \hat{\sigma_1}=\begin{bmatrix} 0&1&0&1  \\ 1&1&1&1 \\ 1&0&0&1 \\ 0&1&1&1 \end{bmatrix}, \
    \hat{\sigma_2}=\begin{bmatrix}  0&1&1&0 \\ 1&1&0&1 \\ 0&1&0&1 \\ 1&1&1&1  \end{bmatrix}, \
    \hat{\sigma_3}=\begin{bmatrix}  1&0&0&0 \\ 0&1&0&0 \\  0&1&1&0\\ 1&1&0&1 \end{bmatrix}$.

Now notice that by Lemmas \ref{lem:rs-on-abcd} and \ref{lem:sigmas-on-abcd}, we have: $n(r_i(w))=\hat{\sigma_i}(n(w))$. This gives the induction step: applying move $i$ takes
$$w \mapsto r_i(w) = w', \qquad \vv \mapsto \sigma_i(\vv) = \vv'.$$
We want to show that $n(w') = 2m(\vv')$.

We have $n(w')=n(r_i(w))=\hat{\sigma_i}(n(w))$, and
$2m(\vv')=2m(\sigma_i(\vv))=2\hat{\sigma_i}(m(\vv))$.
By the induction hypothesis, $n(w) = 2m(\vv)$, so by linearity of $\hat{\sigma_i}$, we can conclude that $n(w') = 2m(\vv')$.
\end{proof}

\subsection{Summary}

We now state what we know about periodic trajectories on the double pentagon surface and the pentagon billiard table, and say how to determine this information from the direction:

\begin{theorem}\label{thm:hecke}
  For a given periodic direction on the golden L, we can give all of the following information about the associated double pentagon and pentagon billiard trajectories, as follows.

  Find the long saddle connection vector $\mathbf{v} = [a+b\phi,c+d\phi]$, using Theorem \ref{thm:abcd}.

  \begin{enumerate}
    \item The combinatorial period of the trajectory on the double pentagon surface is $2(a+b+c+d)$.
    \item Transform $\mathbf{v}$ into a saddle connection vector on the double pentagon surface (using the matrix $P$ from Definition \ref{def:stretch}). The Euclidean length of the trajectory on the double pentagon surface is the length of the resulting vector $P\mathbf{v}$.
  \end{enumerate}

  Find the tree word associated to $\mathbf{v}$, such as $120$. This tells us the product of $\sigma_i$s that takes $[1,0]$ to $\mathbf{v}$. Starting at $A$ in the graph in Figure \ref{colorgraph}, follow \emph{backward} the arrows associated to each $\sigma_i$ in the product, read \emph{backward}, and determine the ending node in the graph.

  \begin{enumerate}
    \item[3.] If the ending node is $A$, $B$, $C$, $D$, or $E$, then the trajectory has the full symmetry of the dihedral group of the regular pentagon, and  the combinatorial period, and the euclidean length, of the pentagon billiard trajectory are five times that of the associated double pentagon trajectory computed above (except for the two trajectories parallel to the edges of the table through midpoints, in which case they are 2.5 times that above).
    \item[4.] On the other hand, if the ending node is $F$,  then the trajectory has only bilateral symmetry, and its combinatorial period, and the euclidean length, of the pentagon billiard trajectory are the same as the associated double pentagon trajectory computed above.
  \end{enumerate}

\end{theorem}

\begin{proof}
  \begin{enumerate} Part (1) is Theorem~\ref{thm:double}; part (2) is Proposition~\ref{prop:scvectors}, and part (3) is Theorem~\ref{thm:symm} and Proposition~\ref{prop:corecurve}, and (4) is Theorem~\ref{thm:symm}.
  \end{enumerate}
\end{proof}

\section{Applications and future directions}\label{sec:further}

Our \verb|Sage| program computes the periodic trajectory for the golden L surface, the double pentagon surface and the pentagon billiard table, associated to any tree word. As a result of exploring these trajectories, we have additional results, conjectures and questions.

\subsection{Combinatorial periods achieved on the surface and table}

By symmetry, all periodic trajectories on the double pentagon have periods that are multiples of 2, and all periodic trajectories with rotational symmetry on the pentagon billiard table have periods that are multiples of 10. It turns out that in both cases,  \emph{every} such number does in fact arise as a periodic trajectory period, via simple constructions.

\begin{proposition}\label{prop:evens}
  Every positive even number arises as the combinatorial period of a double pentagon trajectory.
\end{proposition}

\begin{proof}
  A horizontal trajectory has combinatorial period 2, and the trajectory with tree word $10^n$ has period $2n+2$. Geometrically, we start with the period-4 trajectory whose tree word is 1, and do repeated horizontal Dehn twists; each twist increases the period by 2 (see Figure \ref{fig:all_evens}). Algebraically, by Theorem \ref{thm:abcd}, the trajectory corresponding to tree word $1\cdot 0^n$ has corresponding long saddle connection vector $\sigma_0^n \sigma_1 \sv 10 = \sv {(n+1)\phi} 1$, so by Theorem \ref{thm:double}, the period is $2((n+1)+1) = 2(n+2)$.
\end{proof}

\begin{figure}[!ht]
  \begin{center}
    \includegraphics[width=0.19\textwidth]{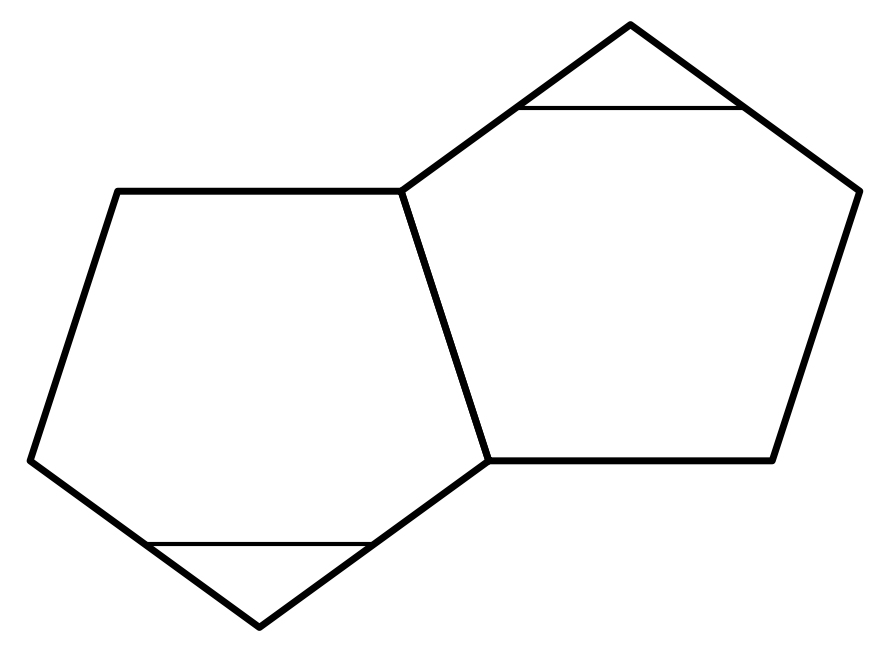}
    \includegraphics[width=0.19\textwidth]{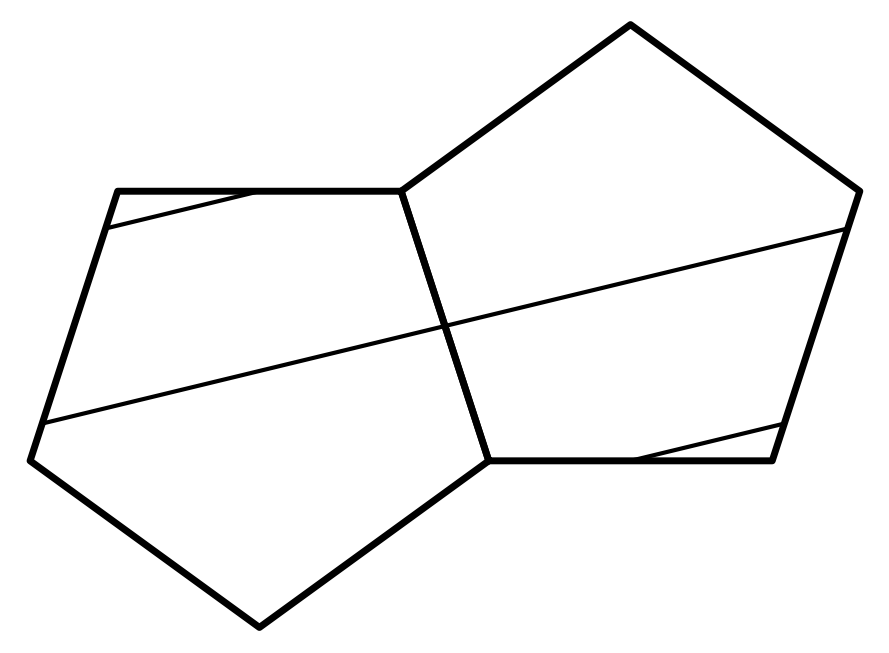}
    \includegraphics[width=0.19\textwidth]{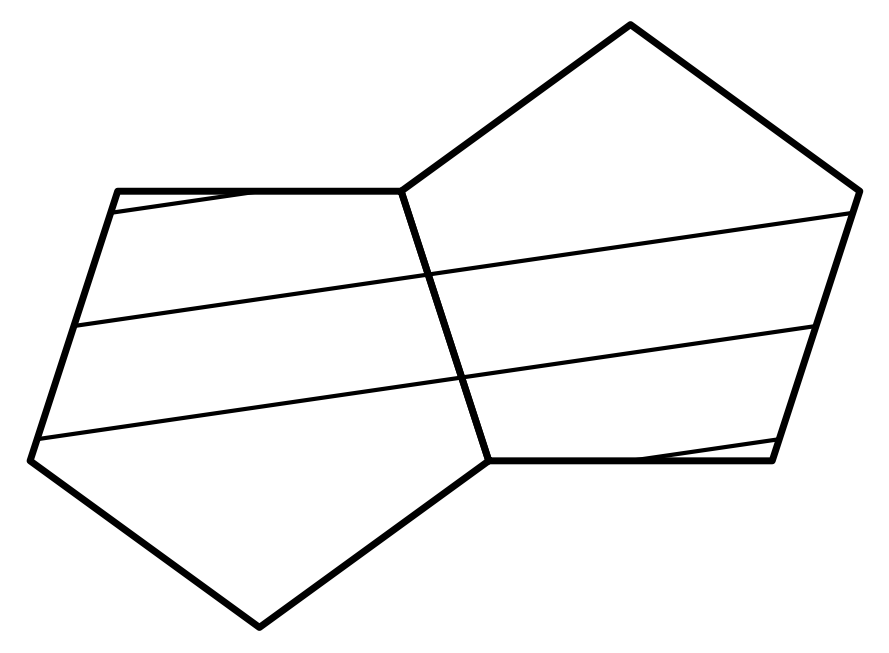}
    \includegraphics[width=0.19\textwidth]{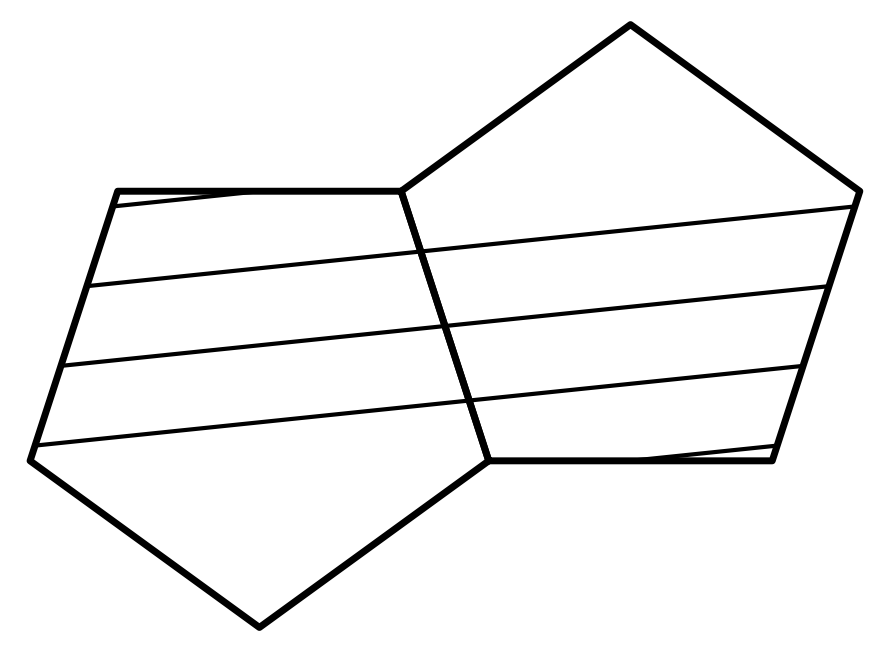}
    \includegraphics[width=0.19\textwidth]{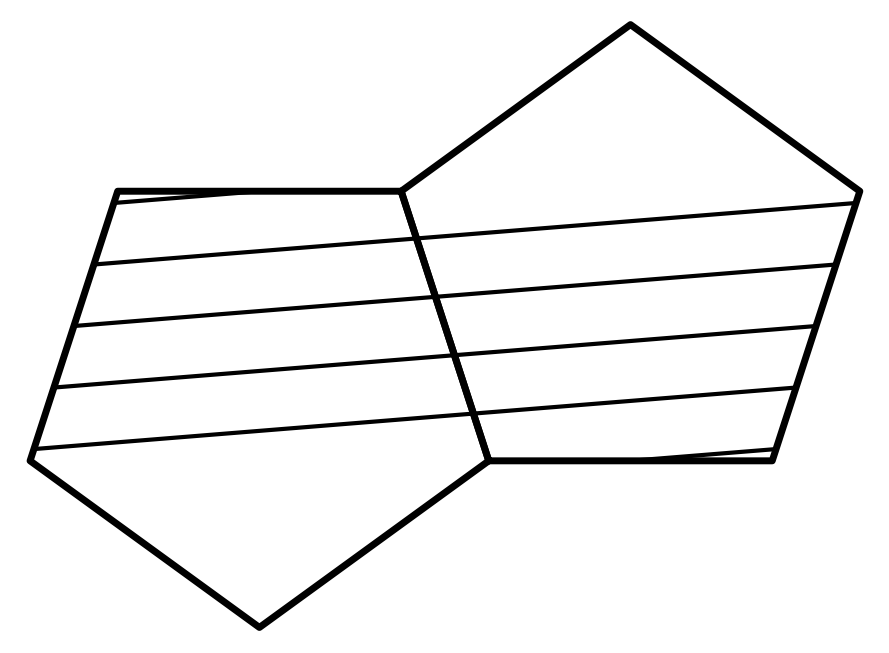}
    \caption{Short periodic paths on the double pentagon, with combinatorial period 2, 4, 6, 8, 10, respectively, corresponding to tree words 0, 1, 10, 100, 1000, respectively. Continuing this construction leads to paths with every even period. \label{fig:all_evens}}
  \end{center}
\end{figure}

\begin{proposition}\label{prop:10}
  Every positive multiple of 10 arises as the combinatorial period of a rotationally symmetric pentagon billiard trajectory.
\end{proposition}

\begin{proof}
  Given $k\in\N$, the trajectory corresponding to the following tree word has period $10k$:
  \begin{enumerate}
    \item If $k\not\equiv 3$ (mod 5), 1 followed by $k-2$ 0s, short trajectory. 
    \item If $k\equiv 3$ (mod 10), 1 followed by $(k-3)/2$ 0s, long trajectory. 
    \item If $k\equiv 8$ (mod 10), 2 followed by $(k-2)/2$ 0s, short trajectory. 
  \end{enumerate}
  These follow from using the directed graph in Figure~\ref{colorgraph} and the Combinatorial Period Theorem~\ref{thm:double}. For case (1), if $k\not\equiv 3$ (mod 5), then $k-2 \not\equiv 1$ (mod 5). We start at $A$ and follow $\sigma_0$ backwards $k-2$ times, landing on $A,B,C$ or $D$. Then we follow $\sigma_1$ backwards, landing on $D,A,E$ or $A$, respectively, all of which correspond to a path with rotational symmetry. (Notice that if $k\equiv 3$, this process results in landing on $F$, yielding a path without rotational symmetry.)

  The direction vector ${\sigma_0}^{k-2} \sigma_1 \sv 1 0 = \sv{(k-1)\phi}{\phi}$, so since the path has rotational symmetry, the period is $10k$, as desired. The other cases are similar.

\end{proof}

Based on computer evidence, we conjecture the following:

\begin{conjecture} Every even number arises as the combinatorial period of an asymmetric periodic billiard trajectory on the pentagon, other than 2, 12, 14 and 18.
\end{conjecture}

We have verified this up to period 4000.

\begin{conjecture}
  For trajectories with only reflection symmetry, the growth rate is linear: There are asymptotically $n/4$ trajectories that have combinatorial period $2n$.
\end{conjecture}

\subsection{Non-equidistribution of periodic billiard trajectories}\label{sec:non-equi}

A given periodic trajectory consists of a finite number of line segments, and thus by definition does not equidistribute. However, we can ask if a sequence of longer and longer periodic trajectories equidistributes in the limit. Most long trajectories do ``equidistribute'' in this sense, but some do not (Figure \ref{fig:non-equi}).

\begin{figure}[!ht]
  \centering
  \includegraphics[width=0.45\textwidth]{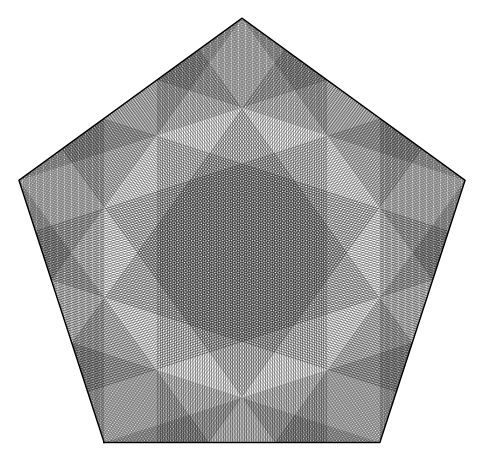} \ \
  \includegraphics[width=0.45\textwidth]{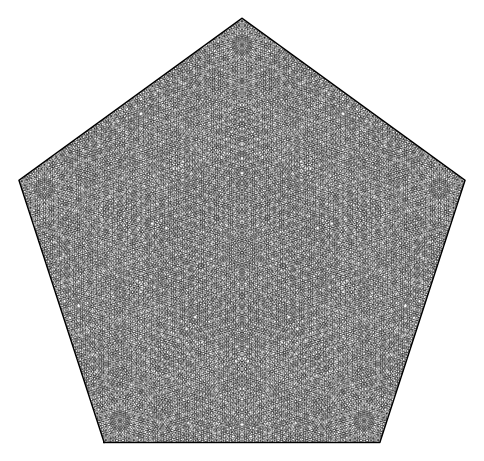}
\caption{Two periodic paths on the pentagon: (a) The short path corresponding to tree word 1000000000000000000000000000000000002, with combinatorial length 1460 and geometric length $\sim$909. (b) The short path corresponding to tree word 12121, with combinatorial length 1470 and geometric length $\sim$964. Both are drawn with the same line thickness, and correspond to the core curve of the cylinder. \label{fig:non-equi}}
\end{figure}


\begin{remark}\label{rem:pudding}
  The only way to get non-equidistributed periodic paths like Figure \ref{fig:non-equi}(a) is with long strings of horizontal ($\sigma_0$) or vertical ($\sigma_3$) Dehn twists. Any initial product of $\sigma_i$s determines how much trajectory is in each cylinder, the product of $\sigma_0$ or $\sigma_3$ twists the horizontal or vertical cylinder to equidistribute the trajectory within the cylinder, and then any ending product of $\sigma_i$s transforms the cylinder direction out of horizontal. See Equation \eqref{eq:pudding} and Figure \ref{fig:pudding}. We thank Barak Weiss for suggesting this construction.
\end{remark}

\begin{equation}\label{eq:pudding}
  \text{tree word} = \overbrace{1}^{\hidewidth\text{controls amount of trajectory in each cylinder}}\underbrace{00000000000000000000000000000000000}_{\text{horizontal Dehn twists}}\overbrace{2}^{\hidewidth\text{transforms cylinder direction}}
\end{equation}

\begin{figure}[!ht]
  \centering
  \includegraphics[width=0.3\textwidth]{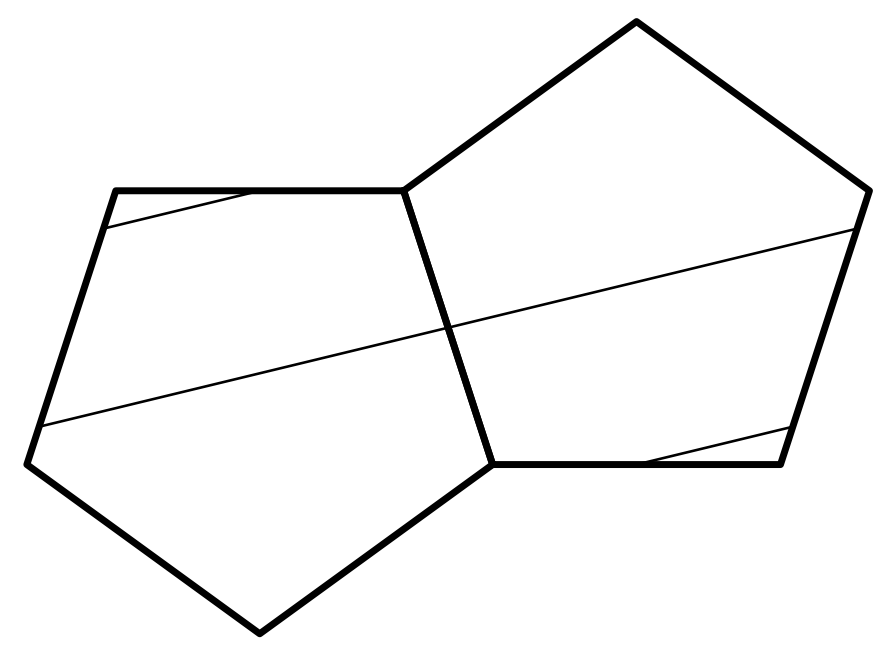} \ \
  \includegraphics[width=0.3\textwidth]{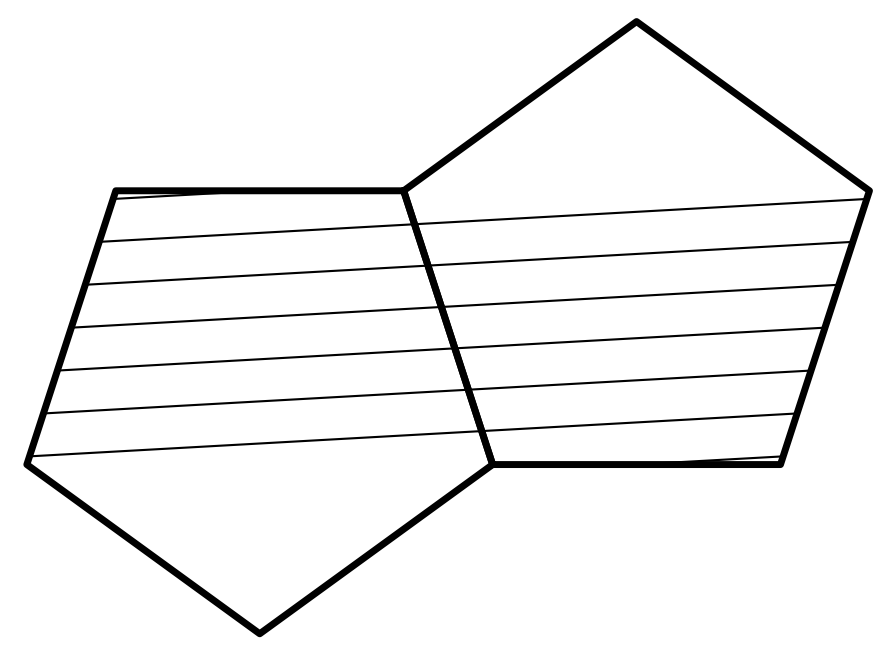} \ \
  \includegraphics[width=0.3\textwidth]{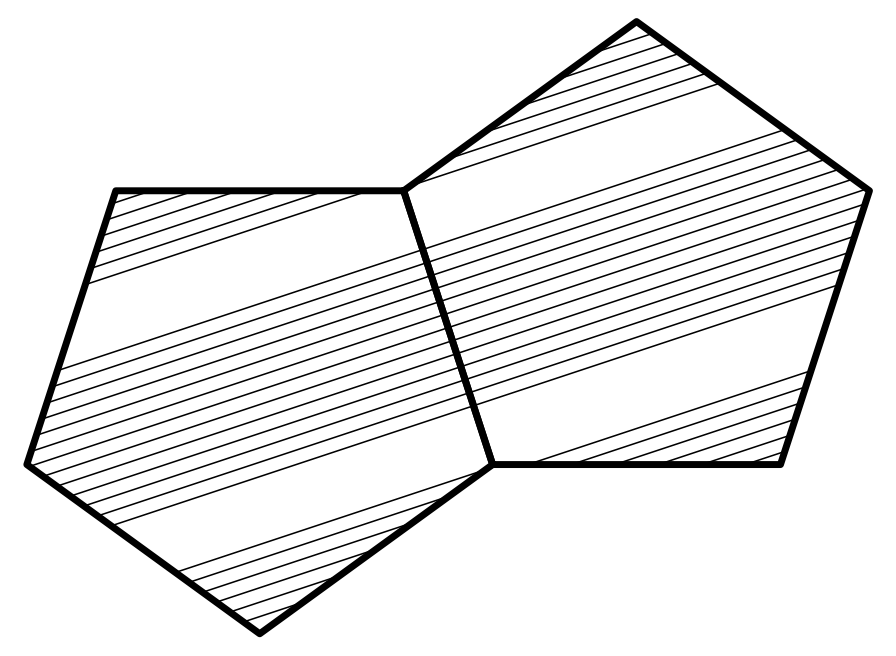}
  \caption{The method behind Figure \ref{fig:non-equi}(a), shown for a simpler example: The short periodic paths on the double pentagon corresponding to tree words 1, 100000, and 1000002, respectively, illustrating Remark \ref{rem:pudding} and Equation \eqref{eq:pudding}. To get the billiard trajectory picture, fold the third double pentagon picture above so that the two pentagons align, and then superimpose four more copies of it on top, in the four other orientations.  \label{fig:pudding}}
\end{figure}

In the family of periodic trajectories with tree words of the form $1000\ldots$ (shown in Figure \ref{fig:all_evens}), the ratio of the ``trajectory concentration'' (length of trajectory per unit area) in the small and large horizontal cylinder is 0; in a nearly equidistributed trajectory, it is near 1. For each finite word $w\in\{0,1,2,3\}^*$, look at the limit of trajectories of the form $w\ 0^n$, which look like horizontal cylinder decomposition pictures with shaded cylinders, because the trajectory equidistributes within each horizontal cylinder (converges to Lebesgue measure) as $n\to\infty$. The darkness of the shading is proportional to the ratio of line length to area in each horizontal cylinder, and depends only on $w$. For most trajectories, the ratio of shading is close to 1. We conjecture that:

\begin{conjecture}\label{conj:mcm}
  The ratios of trajectory concentrations in the horizontal cylinders for a trajectory corresponding to a tree word of the form $w\ 0^n$ forms a Cantor set.
\end{conjecture}

McMullen has since shown Conjecture \ref{conj:mcm} to be false. In fact, the set is more complicated than a Cantor set, and is homeomorphic to $\omega^\omega+1$; see \cite{mod}.

The double pentagon and the necklace are Veech surfaces, and thus exhibit \emph{optimal dynamics}: every trajectory is either periodic or equidistributed. As we saw in Figure \ref{fig:non-equi}, it is possible to find periodic {billiard} trajectories that are arbitrarily dense, but \emph{not} equidistributed. In fact, we can find such periodic trajectories that never visit some part of the table \emph{at all}. Figure \ref{fig:gap} shows examples of two families we have found that exhibit this surprising behavior.

\begin{figure}[!ht]
  \centering
  \includegraphics[width=0.48\textwidth]{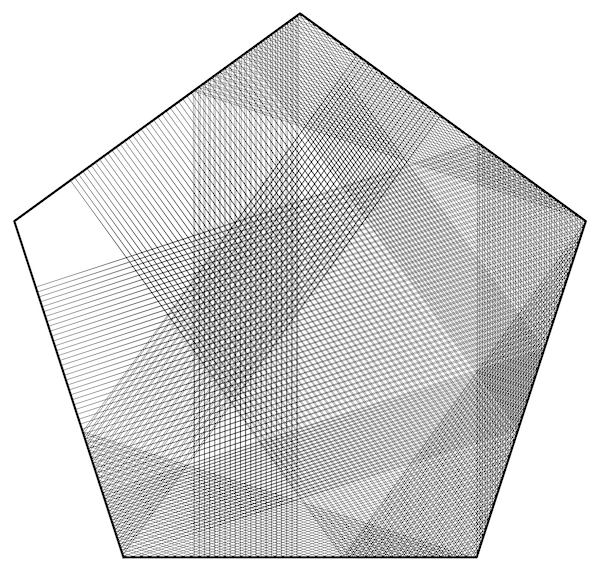} \ \ 
  \includegraphics[width=0.48\textwidth]{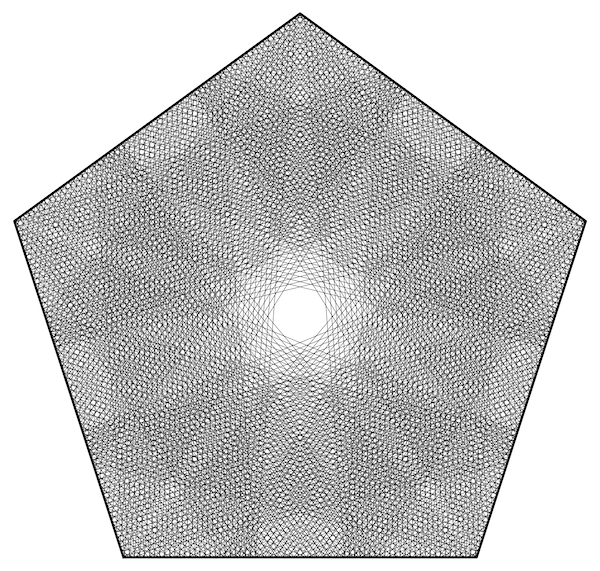} \ \ 
  \caption{Representatives (with $n=20$) of the families of short trajectories with tree words of the form $123^n 1$ and $10^n 1$, respectively, which all miss part of the table. Another family is very similar to the family on the left, and has tree words of the form $303^n1$. \label{fig:gap}}
\end{figure}

\subsection{More conjectures and questions}

In this section we list some more questions that arose out of our work. See also the earlier Question \ref{q:five-sixths}.

\begin{question}
  Some billiard trajectories appear to have ``holes,'' for example the one on the right in Figure~\ref{fig:buddies}. What can we say about the holes?
\end{question}

\begin{question}
  (Curtis McMullen) If you pick an edge midpoint and a periodic direction, how can you tell if it will give a short or a long trajectory?
\end{question}

Our Summer@ICERM 2021 research students subsequently gave an algorithm answering this question \cite[Algorithm 3.2]{icerm21}.

\begin{question}
  So far, the structure of the set of periodic directions and periodic trajectories is only well understood for the rectangle and regular pentagon billiard tables, and possibly the regular hexagon. What is its structure for other tables, such as the regular heptagon and octagon?
\end{question}

\begin{question}
  What is the relationship between  geometric vs.\ combinatorial lengths of saddle connection vectors, and what are the statistics of their distribution?
\end{question}

\begin{question}
  Is there a substitution to go from the short to the long billiard word in a given periodic direction?
\end{question}


Let $p(n)$ be the complexity of the billiard language of a billiard table. Julien Cassaigne, Pascal Hubert, and Serge Troubetzkoy proved that the language complexity $p(n)$ for the billiard language of a convex polygon is always asymptotic to a cubic \cite{cht}. They gave formulas for the limit $\lim_{n\to\infty} \frac{p(n)}{n^3}$ in the square $\left(\frac{4}{\pi^2}\right)$, equilateral triangle $\left(\frac{2}{3\pi^2}\right)$ and isosceles right triangle $\left(\frac{3}{4\pi^2}\right)$. At the end of \cite[\S 1]{cht}, the authors remark that ``It would be interesting to know if the limit $\lim_{n\to\infty} \frac{p(n)}{n^3}$ exists in the case of Veech polygons.'' Our computational evidence suggests that:

\begin{conjecture}
  For the regular pentagon billiard table, we have: $\displaystyle\lim_{n\to\infty}\frac{p(n)}{n^3} = \frac{10}{\pi^2}.$
\end{conjecture}

\begin{question}
  Corollary \ref{cor:sixths} shows that, at a given depth $n$ in the tree, paths equidistribute with respect to landing on elements of the Veech group in the graph in Figure \ref{colorgraph}. Is the same true if we measure with respect to combinatorial length, or geometric length, instead of tree word length?
\end{question}




\begin{question}
  What is the language of the bounce sequences, of sides of the pentagon that the trajectory hits? Note that this is different from the language of cutting sequences on the double pentagon surface, because the labels of corresponding edges are different. It would be interesting to study valid bounce sequences, in the method of Smillie and Ulcigrai \cite{SU}.
\end{question}

\begin{question}
  (See \cite{leutbecher}) There is a simple test to see if a matrix is in the Veech group of the square torus: it must have integer entries and determinant 1. Is there are similarly simple characterization for matrices in the Veech group of the double pentagon or golden L?
\end{question}

\begin{remark}\label{rem:fractal}
The continued fraction algorithm outlined in \S~\ref{sec:cont_frac} involves arithmetic in $\Q(\sqrt{5})$, a quadratic extension of $\Q$, which can naturally be seen as a 2-dimensional vector space over $\Q$. The continued fraction algorithm actually acts on $\Z[\phi]^2$, and involves four $2\times 2$ matrices $\sigma_i$ with coefficients in $\Z[\phi]$. It is natural to see $\Z[\phi]^2$ as a 4-dimensional $\Z$-module, and to view the algorithm as operating in 4 dimensions. This is what the matrices $\hat{\sigma}_i$ from \S~\ref{subsec:extended} do. A natural domain for the corresponding dynamical system is the locus where, at each step, we can apply at least one of the $\hat{\sigma}_i^{-1}$ and remain in the positive cone: in other words, the intersection as $k\to\infty$ of the union of images of the positive cone of $\mathbf{R}^4$ under products of length $k$ of matrices among the $\hat{\sigma}_i$. This is a fractal object living in the 4-dimensional positive cone. It is invariant under scaling, so we take a 3-dimensional section where the sum of the coordinates is 1, which is a subset of a 3-dimensional simplex. At each step $k$, it is a union of tetrahedra meeting along edges (Figure \ref{fig:polytope}).

This process is reminiscent of the Rauzy gasket one dimension down: The \emph{Rauzy gasket} is a 2-dimensional section of a scale-invariant 3-dimensional object, obtained as a limit of objects which, at each step, are a union of triangles meeting at vertices. The Rauzy gasket is a subject of active research towards a solution to the Novikov problem \cite[Figure 2]{rauzy}.

John Smillie and Corinna Ulcigrai made a very similar picture associated to the regular octagon surface that they studied in \cite{SU}, which Smillie showed in a talk at Banff in February 2011. The picture did not appear in their paper, but they shared it with us after seeing the picture below. We include the picture here as this fractal seems to be a natural object, and we hope to inspire future study of this and related objects.\footnote{If everyone keeps brushing this thing under the rug, it's going to become a trip hazard.}
\end{remark}

\begin{figure}[!ht]
  \begin{center}
    \includegraphics[width=0.8\textwidth]{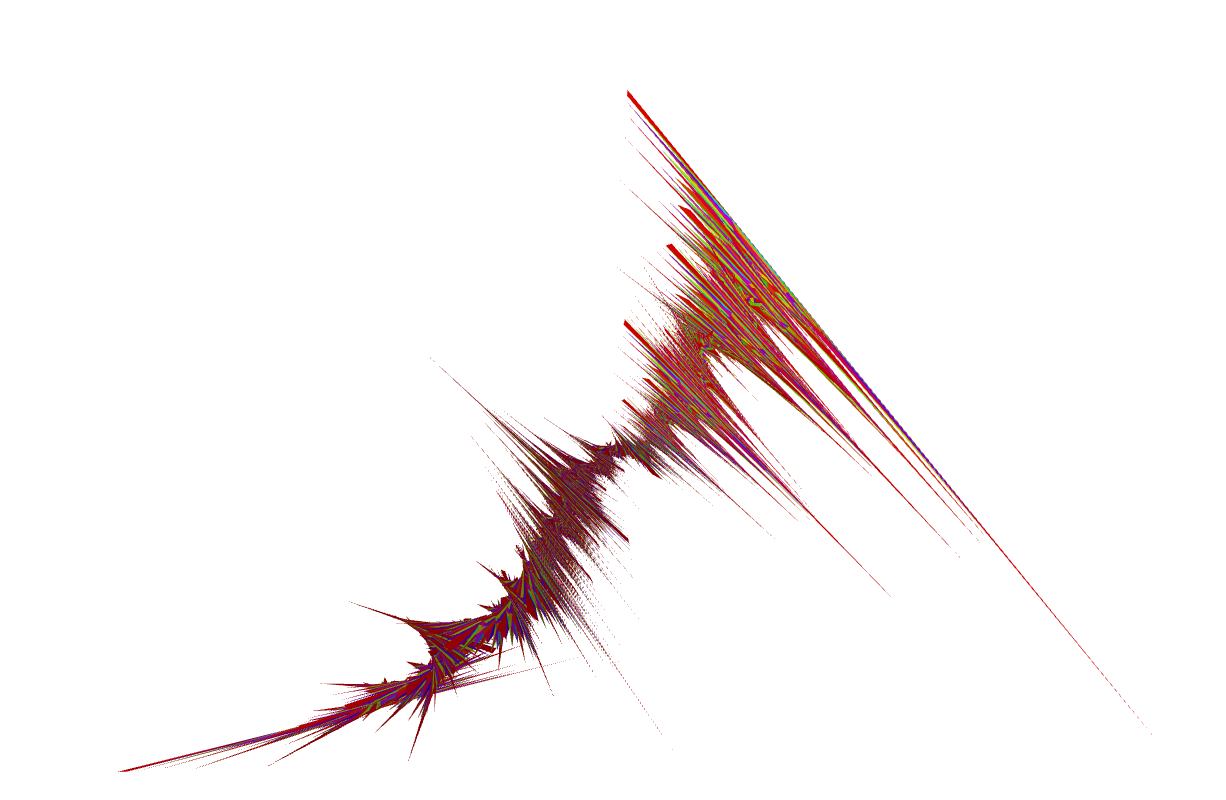}
    \caption{The seventh-level approximation to the fractal obtained via the iterated function system of cone contractions described in Remark \ref{rem:fractal}. \label{fig:polytope}}
  \end{center}
\end{figure}

\begin{question}\label{q:fractal}
  What is the fractal dimension of this object, and what properties does it have?
\end{question}


\end{document}